\documentclass[reqno, 11pt]{amsart}

\usepackage[letterpaper,hmargin=1in,vmargin=1.2in]{geometry}
\usepackage{amsmath, amssymb, amsthm, verbatim, url}
\usepackage{graphicx}
\usepackage{enumerate, pinlabel}
\usepackage{amsmath,amssymb,amsthm,mathrsfs,graphicx,url}
\usepackage[usenames,dvipsnames]{color}
\usepackage[colorlinks=true,linkcolor=Black,citecolor=Black, urlcolor=Black]{hyperref}

\newcommand \tR {R_D}
\DeclareMathOperator \re {Re}
\DeclareMathOperator \im {Im}

\DeclareMathOperator \supp {supp}

\DeclareMathOperator \Op {Op}

\DeclareMathOperator \vol {vol}
\DeclareMathOperator \sgn {sgn}

\newtheorem{prop}{Proposition}
\newtheorem{lem}[prop]{Lemma}

\newtheorem{cor}[prop]{Corollary}
\newtheorem{thm}[prop]{Theorem}
\theoremstyle{definition}
\newtheorem{ex}{Example}
\numberwithin{equation}{section}
\numberwithin{prop}{section}
\numberwithin{figure}{section}

\def\arXiv#1{\href{http://arxiv.org/abs/#1}{arXiv:#1}}

\def \Real {{\mathbb R}}

\def \Natural {{\mathbb N}}
\def \zhath {\hat{Z}_h}

\DeclareMathAlphabet{\mathpzc}{OT1}{pzc}{m}{it}
\def \pr {\mathpzc{p}}

\author{T. J. Christiansen}
\address{Department of Mathematics, University of Missouri, Columbia, MO, USA}
\email{christiansent@missouri.edu}
\author{K. Datchev}
\address{Department of Mathematics, Purdue University, West Lafayette, IN, USA}
\email{kdatchev@purdue.edu}

\title[Resolvent estimates on cylindrical manifolds and on the half line]{Resolvent estimates on asymptotically cylindrical manifolds and on the half line}

\begin{document}

\begin{abstract}

Manifolds with infinite cylindrical ends have continuous spectrum of increasing multiplicity as energy grows, and in general embedded resonances (resonances on the real line, embedded in the continuous spectrum) and embedded eigenvalues can accumulate at infinity. However, we prove that if geodesic trapping is sufficiently mild, then the number of embedded resonances and eigenvalues is finite, and moreover the cutoff resolvent is uniformly bounded at high energies. We obtain as a corollary the existence of resonance free regions near the continuous spectrum.

 We also obtain improved  estimates when the resolvent is cut off away from part of the trapping, and  along the way we prove some resolvent estimates for repulsive potentials on the half line which may be of independent interest. 

\end{abstract}
\maketitle

\section{Introduction}

\subsection{Resolvent estimates for manifolds with infinite cylindrical ends}\label{s:introex}

The high energy behavior of the Laplacian on a manifold of infinite volume is, in many situations, well known to be related to the geometry of the \textit{trapped set}; this is the set of bounded maximally extended geodesics. In the best understood cases, such as  when the manifold has asymptotically Euclidean or hyperbolic ends (see \cite[\S3]{z:sur} for a recent survey),  the trapped set is compact. Some results have been obtained for more general trapped sets (e.g. manifolds with cusps were studied in \cite{cv}) but less detailed information is available.

In this paper we study manifolds with infinite asymptotically cylindrical ends, which have noncompact trapped sets. A motivation for this study comes from waveguides and quantum
dots connected to leads. The spectral geometry of these is closely related to that of asymptotically cylindrical manifolds, and they appear in certain models of
electron motion in semiconductors and of propagation of electromagnetic and sound waves. We give just a few pointers to the physics and applied math literature here  \cite{lcm,raichel, rbbh,ek,bgw}. In \cite{cdwg}, we prove analogues of some of the results below for suitable (star-shaped) waveguides.

The fundamental example of a manifold with cylindrical ends is the Riemannian product $\mathbb R \times \mathbb S^1$, which has an unbounded trapped set consisting of the circular geodesics. We are interested in the behavior of the resolvent of the Laplacian (and its meromorphic continuation, when this exists) for perturbations of such cylinders and their generalizations. As we discuss below, this behavior can sometimes be very complicated, but we show that if some geometric properties of the manifold are favorable, then the resolvent is uniformly bounded at high energy. In the companion paper \cite{cd}, we study the closely related problem of long time wave asymptotics on such manifolds.

We begin with an illustration of a more general theorem to follow, by stating a high energy resolvent estimate for two kinds of \textit{mildly trapping} manifolds $(X,g)$ with infinite cylindrical ends.

\begin{ex}\label{ex:non1}
Let $(r,\theta)$ be polar coordinates in $\mathbb R^d$ for some $d\ge 2$, and let
\[
  X = \mathbb R^d, \qquad g_0 = dr^2 + F(r) dS,
\]
 where $dS$ is the usual metric on the unit sphere, $F(r) = r^2$ near $r=0$, and $F'$ is compactly supported on some interval $[0,R]$ and positive on $(0,R)$; see Figure \ref{f:cigar}. 
\begin{figure}[h]
\hspace{3cm}
\includegraphics{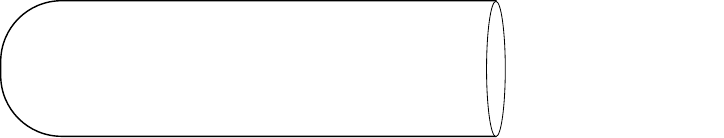}
 \caption{A cigar-shaped warped product.}\label{f:cigar}
\end{figure}

Then for $r(t)>0$ all $g_0$-geodesics obey
\[ \ddot r(t) := \frac{d^2}{dt^2} r(t) = 2 |\eta|^2 F'(r(t))F(r(t))^{-2}\ge0,\]
 where $r(t)$ is the $r$ coordinate of the geodesic at time $t$ and $\eta$ is the angular momentum. Consequently, the only trapped geodesics 
are the ones with $\dot r(t) \equiv F'(r(t)) \equiv 0$, that is the circular ones in the cylindrical end. This is the smallest amount of trapping a manifold with a cylindrical end can have.

Let $g$ be any metric such that $g-g_0$ is supported in  $\{(r,\theta)\mid r < R\}$, and such that $g$ and $g_0$ have the same trapped geodesics. 
For example we may take $g = g_0 + c g_1$, where $g_1$ is any symmetric two-tensor with support in $\{(r,\theta)\mid r < R\}$, and $c \in \mathbb R$ is chosen sufficiently small depending on $g_1$. 
Alternatively, we may take $g = dr^2 + g_S(r)$, where $g_S(r)$ is a smooth family of metrics on the sphere such that $g_S(r) = r^2dS$ near $r=0$ and $g_S(r) = F(r)dS$ near $r \ge R$, and such that $\partial_r g_S(r) > 0$ on $(0,R)$. This way we can construct examples where $g-g_0$ is not small. 

\end{ex}

\begin{ex}\label{ex:3fun1}

Let $(X,g_H)$ be a convex cocompact hyperbolic surface, such as the symmetric hyperbolic `pair of pants' surface with three funnels depicted in Figure \ref{f:3fun}.
\begin{figure}[h]
\labellist
\small
\pinlabel $r$ [l] at 280 -3
\pinlabel $\cosh^2\!r$ [l] at 225 70
\pinlabel $F(r)$ at 275 36
\endlabellist
 \includegraphics[width=5cm]{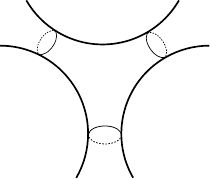} \hspace{2cm} \includegraphics[width=7cm]{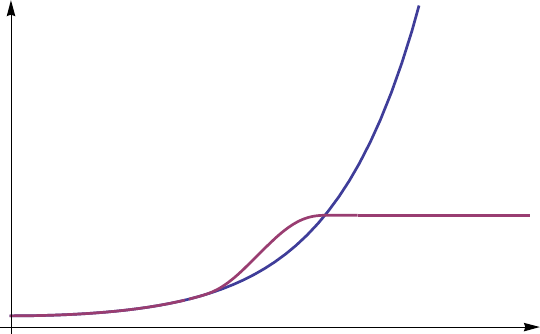}
 \caption{A hyperbolic surface $(X, g_H)$ with three funnels, and a modification of the metric which changes the funnel ends to cylindrical ends.}\label{f:3fun}
\end{figure}

In particular, there is a compact set $N \subset X$ (the convex core of $X$) such that 
\[
 X\setminus N = (0,\infty)_r\times Y_y, \qquad g_H|_{X \setminus N} = dr^2 + \cosh^2 \!r\, dy^2,
\]
where $Y$ is a disjoint union of $k\ge1$ geodesic circles (possibly having different lengths).

We modify the metric in the funnel ends so as to change them into cylindrical 
ends in the following way. Take $g$ such that 
\[
g|_N = g_H|_N, \qquad g|_{X \setminus N} = dr^2 + F(r)dy^2,
\] where $F(r) = \cosh^2\!r$ near $r=0$, and  $F'$ is compactly supported and positive on the interior of the convex hull of its support.

To obtain higher dimensional examples, we can take $(X,g_H)$ to be a conformally compact manifold of constant negative curvature, with dimension $d \ge 3$, but in this case we need the additional assumption that the dimension of the limit set is less than $(d-1)/2$. The construction of $g$ now becomes  more complicated and we give it in \S\ref{s:hypex} below.

\end{ex}

Our first result concerns only the above examples.

\begin{thm}\label{t:resestx}
 Let $(X,g)$ be as in Example \ref{ex:non1} or \ref{ex:3fun1}  above, and let $\Delta \le 0$ be its Laplacian. There is $z_0 >0$ such that for any $\chi \in C_c^\infty(X)$ there is $C>0$ such that
\begin{equation}\label{e:mainex}
 \|\chi(-\Delta - z)^{-1}\chi\|_{L^2(X) \to L^2(X)}\le C,
\end{equation}
for all $z \in \mathbb C$ with $\re z \ge z_0$ and $\im z \ne 0$.
\end{thm}

Here $(-\Delta - z)^{-1}$ denotes the standard resolvent which maps $L^2(X) \to L^2(X)$, and not its meromorphic continuation. Below, in Theorem \ref{thm:resonancefree}, we also obtain bounds for the meromorphic continuation, but these are more complicated to state.

The bound \eqref{e:mainex} is optimal in the sense that we cannot replace the right hand side by a function of $z$ which tends to $0$ as  $\re z \to \infty$. Indeed, taking the case of Example \ref{ex:non1} with $d=2$ for definiteness, we have $(-\Delta - k^2) v(r)e^{ik\theta} = -v''(r)e^{ik\theta}$ for any $v \in C_c^\infty((R,\infty))$ and $k \in \mathbb Z$.

Note also that the resolvent in these examples is better behaved than it is for the (geometrically simpler) Riemannian product $(X,g) = (\mathbb R \times Y, g = dr^2 + g_Y)$, where $(Y,g_Y)$ is a compact Riemannian manifold. Indeed, take $\chi \in C_c^\infty(X)$ a function of $r$ such that $\chi \ge 0$ and $\chi \not\equiv 0$, and take $\chi_0 \in C_c^\infty(X)$ such that $\chi_0\chi = \chi$, and let $\phi$ be an eigenfunction of the Laplacian on $(Y, g_Y)$ with  $-\Delta \phi = \sigma^2 \phi$. Then, by separation of variables,
\begin{equation}\label{e:prodcomp}
\begin{split}
\|\chi(-\Delta - z)^{-1}\chi \chi_0 \phi\|_{L^2(X)} &=  \|\phi\|_{L^2(Y)} \|\chi(-\partial_r^2 - z + \sigma^2)^{-1} \chi\|_{L^2(\mathbb R)} 
\xrightarrow{z \to \sigma^2} +\infty, 
\end{split}\end{equation}
where we take the limit using the explicit formula for the resolvent \cite[(2.2.1)]{dz}. For our proof of Theorem \ref{t:resestx} it will be crucial that $F'>0$ near the cylindrical ends in Examples \ref{ex:non1} and \ref{ex:3fun1}, and this is what is missing in the Riemannian product just discussed.

We will deduce Theorem \ref{t:resestx} from Theorem \ref{t:cont} below, which  gives a stronger result (allowing $\chi$ to be replaced by a noncompactly supported weight) and also applies to Schr\"odinger operators on more general manifolds with asymptotically cylindrical ends.
We will further prove in Theorem \ref{t:away} that we can obtain stronger resolvent bounds by suitably refining the cutoffs $\chi$. 

An estimate like \eqref{e:mainex} has well-known implications for the spectrum of $-\Delta$. In particular, by  \cite[Theorem XIII.20]{rs}, the spectrum is purely absolutely continuous on $(z_0,\infty)$, which rules out any embedded eigenvalues there, and we will see below, in \S\ref{s:continuation}, that embedded resonances (resonances on the real line, embedded in the continuous spectrum) are also ruled~out.

To our knowledge ours is the first result ruling out the presence of infinitely many embedded eigenvalues or resonances for a large class of examples of manifolds with infinite cylindrical ends.

The situation can be very different for other manifolds with cylindrical ends. For example, let $X = \mathbb R \times Y$ and $g=dr^2 + F(r)g_Y$, where $(Y,g_Y)$ is a compact Riemannian manifold and  $F \in C^\infty(\mathbb R;(0,\infty))$, $1-F$ is compactly supported, and $\max F > 1$. Then $-\Delta$ has infinitely many embedded eigenvalues converging to $+\infty$ (\cite[\S 3]{cz}, \cite[(3.6)]{parn}).

The study of the spectral and scattering theory of the Laplacian on manifolds with cylindrical ends, and their perturbations, goes back to Guillop\'e \cite{gui} and Melrose \cite{mel} and is an active and wide-ranging area of research: see for example \cite{ikl, ms, rtda} for some recent results and more references. There is also a large of body of literature on the closely related study of the Laplacian on waveguides: something of a survey can be found in \cite{kk}, and let us also mention the older result \cite{go}, and that there is a nonexistence result for eigenvalues in \cite{dp}. In a slightly different direction, weighted resolvent estimates up to the spectrum  and limiting absorption principles have been investigated using Mourre theory \cite{mourre, abg,dg}, and this has been applied to geometric situations such as ours in \cite{nier}.

Our results also have implications for the distribution of \textit{resonances}; these are the poles of the meromorphic continuation of the resolvent, and their study in this context also goes back to \cite{gui, mel}. An existence result for resolvent poles (in the presence of appropriate quasimodes, and which may be embedded in the real line or complex) on waveguides can be found in \cite{e}, and for more such results see \cite{kk}. Upper bounds on the number resonances for manifolds with infinite cylindrical ends  are given in \cite{c0}.

In Theorem \ref{thm:resonancefree}, we will use an identity due to Vodev \cite{v} to prove that \eqref{e:mainex} (or a more general resolvent estimate up to the spectrum) implies the existence of a resonance free region near the continuous spectrum. In a companion paper to this one, \cite{cd}, we use these results to prove an asymptotic expansion for solutions to the wave equation.

\subsection{Repulsive potentials on the half line}\label{s:introhalf}
In this paper we also obtain some resolvent estimates for Schr\"odinger operators on the half line which we need in the course of the proofs of our main results, and which may  be of independent interest. We state them here.

Let $V_D$ be a bounded, nonnegative, nonincreasing potential on the half line, such that
\begin{equation}\label{e:vmhyp}
V_D'(r) \le - \delta_V (1+r)^{-1} V_D(r) \le 0,
\end{equation}
for some $\delta_V>0$ and for all $r \ge 0$, where if $V_D$ is not everywhere differentiable then \eqref{e:vmhyp} is meant in the sense of measures. Note that in particular the potential is repulsive in the sense of classical mechanics, since $V_D'(r) < 0 $ except where $V_D(r) = 0$.

For $h>0$ and $\zeta \in \mathbb C \setminus [0,\infty)$ let
\[
 (-h^2 \partial_r^2 + V_D - \zeta)^{-1}
\]
denote the Dirichlet resolvent.
In this paper we prove the following semiclassical resolvent estimates:

\begin{thm}\label{t:semi}
 For all $s, \ s_1, \ s_2>1/2$ with $s_1 + s_2 > 2$ there is $C>0$ such that for all  $\zeta \in \mathbb C \setminus [0,\infty)$ and $h>0$ we have
\begin{equation}\label{e:pdbigh}
 \|(1 + r)^{-s}(-h^2 \partial_r^2 + V_D(r) - \zeta)^{-1}(1+r)^{-s}\| \le \frac{C}{h\sqrt{|\zeta|}}, 
\end{equation}
\begin{equation}\label{e:pdsmallh}
 \|(1 + r)^{-s_1}(-h^2 \partial_r^2 + V_D(r) - \zeta)^{-1}(1+r)^{-s_2}\| \le \frac C {h^2},
\end{equation}
and
\begin{equation}\label{e:pdvboundh}
 \|V_D(r)^{1/2}(1+r)^{- 1/2} (-h^2 \partial_r^2 + V_D(r) - \zeta)^{-1}(1+r)^{-s }\|\le \frac C h,
\end{equation}
where the norms are $L^2(\mathbb R_+) \to L^2(\mathbb R_+)$.
\end{thm}

Recall that, in the case $V_D \equiv 0$, \eqref{e:pdbigh} and \eqref{e:pdsmallh} are well known to be sharp as dist$(\zeta,[0,\infty)) \to 0$; this can be checked from the explicit formula for the resolvent in that case, which we give below in \eqref{e:resdhalf}.

In fact, we will deduce these estimates from some uniform estimates for Schr\"odinger operators with repulsive potentials, replacing $C$ by an explicit constant. To state them, let
\[
 P_D := -\partial_r^2 + V_D(r),
\]
regarded as a self-adjoint operator on $L^2(\mathbb R_+)$ with domain $\{u \in H^2(\mathbb R_+) \mid u(0) = 0\}$.

\begin{thm}\label{t:half}
For all $\delta >0$,  $\theta \in [0,1]$, and  $z \in \mathbb C \setminus [0,\infty)$, we have
\begin{equation}\label{e:pdbig}
 \|(1 + r)^{-\frac{1 + \delta}2}(P_D -z)^{-1}(1+r)^{-\frac{1 + \delta}2}\| \le \frac{1 + \sqrt 2}{\sqrt{|z|}}\left( \frac 1 \delta + \frac 1{\delta_V}\right), 
\end{equation}
\begin{equation}\label{e:pdsmall1}
 \|(1 + r)^{-\frac{1 + \delta}2 - \theta}(P_D -z)^{-1}(1+r)^{-\frac{1 + \delta}2 - (1-\theta)}\| \le (1 + \sqrt 2)\left( \frac 1 \delta + \frac 1{\delta_V}\right),
\end{equation}
and
\begin{equation}\label{e:pdvbound}
 \|V_D(r)^{\frac \theta 2}(1+r)^{- \frac {1 +(1-\theta)\delta} 2 }(P_D -z)^{-1}V_D(r)^{\frac{1-\theta}2} (1+r)^{- \frac {1 + \theta \delta} 2 }\|\le \frac {2\sqrt2}{ \delta_V} \sqrt{1+\frac{\delta_V}\delta},
\end{equation}
where the norms are $L^2(\mathbb R_+) \to L^2(\mathbb R_+)$.
\end{thm}

Note that Theorem \ref{t:half} implies Theorem \ref{t:semi}.

If $V_D \in C^1([0,\infty))$ is compactly supported and has $V_D'<0$ on the interior of the support of $V_D$, then \eqref{e:vmhyp} is satisfied for some $\delta_V>0$ (because $\log V_D$ and $(\log V_D)'$ tend to $-\infty$ at the boundary of the support). Moreover the class of potentials satisfying \eqref{e:vmhyp} for a given $\delta_V>0$ is closed under nonnegative linear combinations and contains all functions of the form $(1+r)^{-m}$ with $m\ge \delta_V$. The same proof could also handle potentials $V_D$ satisfying \eqref{e:vmhyp} and such that $V_D(r) \to \infty$ as $r \to 0$, provided $V_D(r)|u(r)|^2 \to 0$ as $r \to 0$ for all $u$ in the domain of $P_D$.

The bounds \eqref{e:pdbigh} and \eqref{e:pdbig} are best when the spectral parameter is not  too close to $0$, and \eqref{e:pdsmallh} and \eqref{e:pdsmall1} are best when the spectral parameter is close to $0$. We can think of \eqref{e:pdvboundh} and \eqref{e:pdvbound} as being a kind of Agmon or elliptic estimate in the limit $|z| \to 0$ (see also \eqref{e:agmonoverlap} below); they give an improvement when we are looking at the resolvent in the elliptic and classically forbidden range in the interior of the support of $V_D$.  When $V_D(r) \sim (1+r)^{-m}$ as $r \to \infty$ for some $m>0$, the weights in \eqref{e:pdvbound} are also to be compared to the weights in \cite{yafbook, nak}; see in particular \cite[Theorem 1.3]{nak}.

If we do not demand explicit constants in the estimates, then Theorem \ref{t:half} is essentially well-known if either $V_D(0)$ (which we can think of as a coupling constant) is not large (see \cite[Chapter 4]{yafbook} for a more general discussion of scattering on the half line, and \cite{kt} for some more recent results and references), or if $V_D(0)$ and $|z|$ are large (this is the semiclassical, nontrapping regime: see \cite[Chapter 7, Theorem 1.6]{yafbook} for a similar result). The main novelty here is that we cover all values of $V_D(0)$ and $|z|$ uniformly, and for our applications in \S\ref{s:mild} we will especially need the case where $V_D(0)$ is large compared to $|z|$: this corresponds to a low-energy semiclassical problem.

We prove Theorem \ref{t:half} in \S\ref{s:halfline} below.

\subsection{Notation}

Throughout the paper $C$ is a large constant which can change from line to line, and all estimates are uniform for $h \in (0,h_1]$, where $h_1$ can change from line to line. It will sometimes be  convenient to write derivatives with respect to $r$ using the notation $D_r:=-i\partial_r$.
We use \[\|u\|_{H^m_h(X)} := \|(-h^2\Delta  + 
1)^{m/2}u\|_{L^2(X)},\]
and similarly define $\|u\|_{H^m_h(\mathbb R)}$ and $\|u\|_{H^m_h(\mathbb R_+)}$ (in the latter case we will only be concerned with $u$ vanishing near $r=0$, so  the boundary condition on the Laplacian implicit in the notation in this case is immaterial).

The energy level $E_0>0$ is fixed in \S\ref{s:assump}, along with the rest of the notation needed for our general abstract setup of a mildly trapping Schr\"odinger operator on a manifold with asymptotically cylindrical ends. The auxiliary notations $E_j$ and $E_*$ are defined in \S\ref{s:xc} in terms of this setup. The notation $E$ without a subscript is used in \S\ref{s:halfline} and \S\ref{s:continuation} to denote a variable positive energy, not related in any particular way to $E_0$ or $E_j$ or $E_*$.

The radial variable $r$ on the cylindrical end has the same meaning in \S\ref{s:assump}, in \S\ref{s:proof}, and in \S\ref{s:continuation}. The usage in \S\ref{s:halfline} is consistent with this usage, if we separate variables to write Schr\"odinger operator on an asymptotically cylindrical end as a sum of Schr\"odinger operators on $\mathbb R_+$. For example, if $\Delta$ is the Laplacian on $((0,\infty)\times Y,dr^2 + g_Y)$ we write
\[
-\Delta = \sum_{j=0}^\infty (-\partial_r^2 + \sigma_j^2) \phi_j \otimes\phi_j, \text{ to mean } -\Delta u = \sum_{j=0}^\infty \phi_j  \int_Y  (-\partial_r^2 + \sigma_j^2) u(r,y) \phi_j(y) d\!\vol (y),
\]
where  $\{ \phi_j\}_{j=0}^\infty$ is a complete set of real-valued orthonormal eigenfunctions of the Laplacian on $Y$ and $-\Delta_Y \phi_j=\sigma_j^2 \phi_j$.

Of course the results of \S\ref{s:halfline} also apply to more general Schr\"odinger operators on $\mathbb R_+$. 

The variable $r$ is used a little differently in \S\ref{s:introex}, \S\ref{s:hypex}, and \S\ref{s:warp}. To convert the $r$ in one of these sections to the $r$ in the rest of the paper, use the affine map
\begin{equation}\label{e:rmapsto}r \mapsto 6(r-R_1)/(R_2-R_1),\end{equation}
for suitably chosen $R_1$ and $R_2$, and then multiply $g$ by $(R_2-R_1)^2/36$ to remove the factor that appears in front of $dr^2$.
For Example \ref{ex:non1}, take $R_1$ such that $\inf\{r>0 \mid g(r,y) = g_0(r,y) \textrm{ for all } y\} < R_1<R$ and use $R_2 = R$. For Example \ref{ex:3fun1}, let $R_2 = \max\supp F'$, and take $R_1 \in (0,R_2)$. For \S\ref{s:hypex}, let $R_1=R+1$ and $R_2 = \max\supp F'$.  For \S\ref{s:warp}, let $R_1 = R/2$ and $R_2=R$.

\setcounter{tocdepth}{1}
\tableofcontents

\section{Resolvent estimates on the half line}\label{s:halfline}

In this section we prove Theorem \ref{t:half}. All function norms and inner products in this section are in $L^2(\mathbb R_+)$, and operator norms are $L^2(\mathbb R_+) \to L^2(\mathbb R_+)$.

\begin{proof}[Proof of \eqref{e:pdbig}]

Let $E := \re z$ and $\varepsilon := |\im z|$. We begin by proving an a priori  estimate when $E>0$ and $\varepsilon>0$. Roughly speaking, the idea is to exploit the fact that, since $V_D' \le 0$, we have the positive commutator  $[P_D,r \partial_r] = - 2  \partial_r^2 - r V_D'(r) \ge 0$. However, to be able to control the remainder terms in our positive commutator argument, we must replace $r \partial_r$ with $w(r)\partial_r$ where $w$ grows more slowly. Such commutants have been used by many authors (see \cite[\S XIII.7]{rs} and references therein); below we take an approach inspired by  \cite{v, d:lim} and papers cited therein.

Take $w \in C^1([0,\infty);[0,1])$ such that $w'(r)>0$ for all $r \ge 0$, 
and take $u \in H^2(\mathbb R_+)$ such that $u(0) = 0$ and $(w')^{-1/2} (P_D - z)u \in L^2$; in particular, $u(r)$ and $u'(r)$  tend to $0$ as  $r 
\to \infty$. Adding together the integration by parts identities
\[
-\langle (w(V_D - E))'u , u \rangle  = 2 \re \langle w(V_D -E) u ,  u'\rangle,
\]
and
\[
 \langle w' u',  u' \rangle + w(0)|u'(0)|^2  = - 2 \re \langle w  u'' ,  u'\rangle,
\]
gives
\[\begin{split}
E \|\sqrt {w'} u\|^2 + \|\sqrt {w'} u'\|^2- \langle (wV_D)'  u,u \rangle  &+ w(0) |u'(0)|^2  =   2 \re \langle w(P_D - z)u,u' \rangle  - 2 \im z \im \langle w u, u' \rangle.
\end{split}\]
Since $0\le w \le 1$, this implies
\begin{equation}\label{e:ewprime}
 E \|\sqrt {w'} u\|^2 + \|\sqrt {w'} u'\|^2 - \langle (wV_D)'  u,u \rangle \le 2\left\|\frac 1 {\sqrt {w'}}(P_D - z)u\right\|\|\sqrt {w'} u'\|  + 2\varepsilon \|u\|\|u'\|.
\end{equation}
Later we will choose $w$ so that $(wV_D)' \le 0$, but first we estimate the second term on the right, which we think of as a remainder term.
Since $V_D \ge 0$, integrating by parts gives
\[
 \|u'\|^2 \le \re \langle(P_D - z)u,u\rangle + E\|u\|^2 \le \left\|\frac 1 {\sqrt {w'}}(P_D - z)u\right\|\|\sqrt {w'} u\| + E\|u\|^2,
\]
and we also have
\[
  \varepsilon\|u\|^2 = | \im \langle(P_D - z)u,u \rangle| \le \left\|\frac 1 {\sqrt {w'}}(P_D - z)u\right\|\|\sqrt {w'} u\|.
\]
Combining these gives
\[
 \varepsilon^2 \|u\|^2\|u'\|^2 \le (E + \varepsilon)\left\|\frac 1 {\sqrt {w'}}(P_D - z)u\right\|^2\|\sqrt {w'} u\|^2,
\]
and then plugging this into \eqref{e:ewprime} gives
\[
  E \|\sqrt {w'} u\|^2 + \|\sqrt {w'} u'\|^2 - \langle (wV_D)'  u,u \rangle \le 2 \left\|\frac 1 {\sqrt {w'}}(P_D - z)u\right\|\left(\|\sqrt {w'} u'\| + \sqrt{E + \varepsilon}\|\sqrt {w'} u\| \right).
\]
Completing the square gives
\begin{equation}\label{e:compsq}\begin{split}
\left(\sqrt E \|\sqrt {w'} u\|  -  \frac{\sqrt{E + \varepsilon}}{\sqrt E} \left\|\frac 1 {\sqrt {w'}}(P_D - z)u\right\|\right)^2  &+ \left( \|\sqrt {w'} u'\| -  \left\|\frac 1 {\sqrt {w'}}(P_D - z)u\right\|\right)^2 \\
- &\langle (wV_D)'  u,u \rangle \le \frac {2E + \varepsilon } {E}\left\|\frac 1 {\sqrt {w'}}(P_D - z)u\right\|^2.
\end{split}\end{equation}
We now take
\begin{equation}\label{e:wdef}w(r) := 1 - \frac {\delta_V}{\delta_V + \delta} (1+r)^{-\delta},
\end{equation}
so that, by \eqref{e:vmhyp}, we have
\begin{equation}\label{e:wvd}
 (wV_D)'(r) = \frac {\delta \delta_V V_D(r) }{(\delta_V + \delta) (1+r)^{1+\delta}}+ w(r)V_D'(r) \le \frac{\delta_V V_D(r)}{1+r}\left((1+r)^{-\delta}  - 1\right) \le 0,
\end{equation}
where, as with \eqref{e:vmhyp}, we understand \eqref{e:wvd} in the sense of measures in the case that $V_D$ is not differentiable everywhere.
We may now drop the second  and third terms from the left hand side  of \eqref{e:compsq}, giving
\begin{equation}\label{e:pdbig0}
 \sqrt E \|\sqrt {w'} u\| \le  \frac{\sqrt{E + \varepsilon} + \sqrt{2E + \varepsilon}} {\sqrt E} \left\|\frac 1 {\sqrt {w'}}(P_D - z)u\right\|.
\end{equation}

From \eqref{e:pdbig0} we can deduce a weighted resolvent estimate when $\re z >0$, $\im z \ne 0$. To obtain an estimate for all $z \in \mathbb C \setminus [0,\infty)$, we use the Phragm\'en--Lindel\"of principle in the following way. For $u, \ v \in L^2(\mathbb R_+)$, put
\begin{equation}\label{e:plu}
 U(z):= \langle(1 + r)^{- \frac {1 + \delta} 2} (P_D - z)^{-1} (1 + r)^{- \frac {1 + \delta} 2} u, v \rangle\sqrt z,
\end{equation}
and for $\alpha>0$ put
\[
 \Omega_\alpha := \{z \in \mathbb C \mid \alpha \re z < |\im z|\}.
\]
Then $U$ is holomorphic in $\Omega_\alpha$, where it obeys
\[
  |U(z)| \le\frac{|\sqrt z| \|u\|\|v\| }{\textrm{dist}(z,[0,\infty))}   \le \frac {\sqrt{1 + \alpha^{-2}} \|u\|\|v\|}{|\sqrt{ z}|}
\]
Moreover, by \eqref{e:pdbig0}, for $z \in \partial \Omega_\alpha \setminus \{0\}$, we have
\begin{equation}\label{e:plalpha}
  | U(z)| \le \left(\sqrt{1 + \alpha} + \sqrt {2 + \alpha}\right)\left(\delta^{-1} + \delta_V^{-1}\right)\|u\|\|v\|.
\end{equation}
Then the Phragm\'en--Lindel\"of principle (see e.g. \cite[p. 236]{rs}) implies \eqref{e:plalpha} for all $z \in \Omega_\alpha$. Taking $\alpha \to 0$ gives \eqref{e:pdbig}.
\end{proof}

\begin{proof}[Proof of \eqref{e:pdsmall1}]
 We begin by following the proof of \eqref{e:pdbig}, but we drop the first term, rather than the second, from the left hand side of \eqref{e:compsq}, so that in place of \eqref{e:pdbig0} we have
\[
 \|\sqrt {w'} u'\| \le \left(1 + \sqrt{2 + \varepsilon E^{-1}}\right)\left\|\frac 1 {\sqrt {w'}}(P_D - z)u\right\|.
\]
We now integrate by parts to obtain a weighted version of the Poincar\'e inequality:
\[
 \left\|(1 + r)^{\frac{-3 - \delta}2} u\right\|^2  = \frac 2 {2 + \delta}\re \left\langle(1+r)^{-2 - \delta}u', u \right\rangle \le \left\|(1+r)^{\frac{-1 - \delta}2}   u'\right\|\left\| (1 + r)^{\frac{-3 - \delta}2}u \right\|,
\]
giving
\begin{equation}\label{e:pdsmall0}
  \left\|(1 + r)^{\frac{-3 - \delta}2} u\right\| \le \sqrt{\delta_V^{-1} + \delta^{-1}}\left(1 + \sqrt{2 + \varepsilon E^{-1}}\right)\left\|\frac 1 {\sqrt 
{w'}}(P_D - z)u\right\|.
\end{equation}
We now apply the Phragm\'en--Lindel\"of principle as in the proof of \eqref{e:pdbig}, with the difference that in place of \eqref{e:plu} we use
\[
U(z):= \langle(1 + r)^{- \frac {3 + \delta} 2} (P_D - z)^{-1} (1 + r)^{- \frac {1 + \delta} 2} u, v \rangle,
\]
to obtain \eqref{e:pdsmall1} when $\theta =1$. Then taking the adjoint gives the result for $\theta=0$, and interpolating (that is to say, applying the Phragm\'en--Lindel\"of principle with respect to $\theta \in \mathbb C$ such that $\re \theta \in [0,1]$) gives the result for $\theta \in (0,1)$.
\end{proof}

\begin{proof}[Proof of \eqref{e:pdvbound}] We again proceed as in the proof of \eqref{e:pdbig}, but this time we replace
\eqref{e:wdef} by
\[
 w(r): = 1 - \frac {\delta_V}{2(\delta_V + \delta)}(1+r)^{-\delta},
\]
so that \eqref{e:wvd} is replaced by
\[
  (wV_D)'(r) \le  -  \frac{\delta_V V_D(r)}{2(1+r)}.
\]
Now dropping the first two terms on the left hand side of \eqref{e:compsq}  gives
\[
 \left\langle\frac{\delta_V V_D(r)}{2(1+r)} u, u\right\rangle \le \frac {2E + \varepsilon } {E}\left\|\frac 1 {\sqrt {w'}}(P_D - z)u\right\|^2,
\]
or
\[
  \|V_D(r)^{\frac 12} (1+r)^{- \frac 12 }(P_D -z)^{-1} (1+r)^{- \frac {1 + \delta} 2 }\|\le \frac { 2\sqrt{2 + \varepsilon E^{-1}}} { \sqrt {\delta_V}}\sqrt{\delta^{-1} + \delta_V^{-1}}.
\]
We now proceed as in the proof of \eqref{e:pdsmall1}, applying the Phragmen--Lindel\"of principle to obtain \eqref{e:pdvbound} for $\theta =1$, and then taking the adjoint and interpolating to obtain \eqref{e:pdvbound} for $\theta \in [0,1)$.
\end{proof}

\section{Resolvent estimates for mildly trapping manifolds}\label{s:mild}

In \S\ref{s:assump} we state our main resolvent estimates for mildly trapping manifolds with asymptotically cylindrical ends, under suitable abstract assumptions. In the remainder of \S\ref{s:mild} we give examples which satisfy the assumptions, and then in \S\ref{s:proof} we prove the estimates.

\subsection{Resolvent estimates for asymptotically cylindrical manifolds}\label{s:assump}

Let $(X,g)$ be a smooth Riemannian manifold of dimension $d \ge 2$, with or without boundary, with the following kind of asymptotically cylindrical 
ends: we assume there is an open set $X_e \subset X$ such that $\partial X \cap X_e = \varnothing$, $X \setminus X_e$ is compact, and
\[
 X_e = (0,\infty)_r \times Y, \qquad g|_{X_e} = dr^2 + f(r)^{4/(d-1)}g_Y.
\]
Here $Y$ is a compact, not necessarily connected, manifold without boundary of dimension $d-1$, $g_Y$ is a fixed smooth metric on $Y$ and $f \in C^\infty([0,\infty);(0,1])$. We suppose further that there is $\delta_0>0$ such that
\begin{equation}\label{e:fdecay}|(f - 1)^{(k)}(r)| \le C_k (1+r)^{-k -\delta_0} \textrm{ for all } k \in \mathbb N_0 \textrm{ and } r \ge 0,\end{equation}
and
\begin{equation}\label{e:fdelta0}f'(r) \ge \delta_0 (1 + r)^{-1}(1 - f) \ge 0 \textrm{ for all } r \ge 0.\end{equation}
Suppose finally that $f(r) < 1$ for $r < 6$. Note that if we replace $r < 6$ by $r < r_0$ in this last condition, we can reduce to the case $r_0=6$ by multiplying $g$ by a constant and rescaling $r$ (i.e. using \eqref{e:rmapsto} with $R_1 = 0$ and $R_2 = r_0$).

We briefly discuss the assumptions \eqref{e:fdecay} and \eqref{e:fdelta0}. Note that the class of functions $f$ such that \eqref{e:fdecay} and \eqref{e:fdelta0} hold for a given $\delta_0>0$ is convex, and contains all functions of the form $f(r) = 1 - (1+r)^{-m}$ whenever $m \ge \delta_0$. Moreover, all functions $f$, such that $f'$ is compactly supported and positive on the interior of the support of $(1-f)$, obey \eqref{e:fdecay} and \eqref{e:fdelta0} for some $\delta_0>0$; indeed, letting $R_f:=\max\supp(1-f)$, we have
\[\lim_{r \uparrow R_f} \log(1 - f(r)) =\lim_{r \uparrow R_f} \frac d {dr} \log(1 - f(r)) =- \infty.\]

If $f'$ is compactly supported then the ends are cylindrical, rather than just asymptotically cylindrical.

 For notational convenience let us  extend $r$ to be a continuous function on $X$ 
with $-1/2\le r<0$ on $X\setminus\overline{X_e}$, and extend $f$ to be constant for $r \le0$.

Let $\Delta \le 0$ be the Laplacian on $X$. Let
\[
 P= P_h:= -h^2 \Delta + V,
\]
where $h \in (0,h_0]$ for some $h_0>0$, and:
\begin{itemize}
 \item $V= V_h \in C^\infty(X \times (0,h_0];\mathbb R)$  is bounded, together with all derivatives, uniformly in $h \in (0,h_0]$.
\item $V|_{X_e}$ is a function of $r$ and $h$ only, and has a decomposition $V|_{X_e} = V_L + h V_S$, where $V_L$ and $V_S$ may also depend on $h$, and $V_S = 0$ for $r \ge 5$ and $|V_S^{(k)}(r)| + |V_L^{(k)}(r)| \le C_k (1+r)^{- k -\delta_0}$ for all $k \ge 0$, uniformly in $h$.
\item $V_L'(r) \le - \delta_0 (1+r)^{-1}V_L(r)\le 0$ for all $r \ge 0$.
\end{itemize}

Note that the assumptions allow $V \equiv 0$ but not $f \equiv 1$. Such a restriction is necessary to obtain a resolvent bound which is uniform up to the spectrum, in light of the computation in \eqref{e:prodcomp}, which rules out such a bound in the case $(X, g) = (\mathbb R \times Y, dr^2 + dS)$ and $P = -h^2 \Delta$.

 Fix $E_0>0$.
We suppose that $E_0$ is a ``mildly trapping'' energy level for $P$ in the sense that adding a complex absorbing barrier 
supported on $X_e$ gives a polynomial resolvent bound. More specifically, 
suppose that for some $W_K \in C^\infty(\mathbb 
R;[0,1])$ with $W_K = 0$ near $(-\infty,5]$  and $W_K = 1$ near $[6,\infty)$, 
there is $N\in \mathbb R$ such that
 \begin{equation}\label{e:w}
  \|(P -i W_K(r) -E_0 )^{-1}\|_{L^2(X) \to L^2(X)} =: a(h)h^{-1} \le
h^{-N},
 \end{equation}
for all $h \in (0,h_0]$.

We have the following weighted resolvent bound up to the spectrum.

\begin{thm}\label{t:cont}
Let $(X,g), \ P, \ E_0,$ and $a(h)$ be as above. Fix $s_1, \ s_2 >1/2$ such that $s_1 + s_2>2$. There are $C>0$ and $h_1>0$ such that
\begin{equation}\label{e:2n-1}
\|(1 + r)^{-s_1}(P - E_0 - i \varepsilon)^{-1}(1+r)^{-s_2}\|_{L^2(X) \to L^2(X)} \le C(a(h)+ h^{-1})h^{-1},
\end{equation}
for all $\varepsilon \in \mathbb R \setminus 0$ and for all $h \in (0,h_1]$.
\end{thm}
Note that the condition on $s_1$ and $s_2$ is the same as the one in \S\ref{s:introhalf} above, see in particular \eqref{e:pdsmallh} and \eqref{e:pdsmall1}. This is the resolvent weighting needed to have a low energy bound for scattering on the half line (and for more general Euclidean scattering problems).

To deduce Theorem \ref{t:resestx} from Theorem \ref{t:cont}, in Examples  \ref{ex:non1} and \ref{ex:3fun1} we let $X_e$ be the part of $X$ where $r \ge r_1$, for any $r_1>0$ such that $F'(r_1)>0$, and put $V \equiv 0$. Then, after redefining $r$ as in the remark following \eqref{e:fdelta0}, we see that $g$ has the desired form in $X_e$, and it remains to check that \eqref{e:w} holds with $N\le2$. Below in \S\ref{s:nonex} and \S\ref{s:hypex} we will show this for some examples which generalize Examples \ref{ex:non1} and  \ref{ex:3fun1} above.

We also have an improved bound when we cut off away from the trapping in the end. To state it, let $\chi_\Pi \in C^\infty(\mathbb R;[0,1])$ be $0$ near $(-\infty,0]$ and $1$ near $[1,\infty)$. Let $\Delta_Y\le0$ be the Laplacian on $(Y, g_Y)$, and let $\{\phi_j\}_{j=0}^\infty$ be a complete real-valued orthonormal set of its eigenfunctions, with $-\Delta_Y \phi_j = \sigma_j^2 \phi_j$, where $0 = \sigma_0 \le \sigma_1 \le \cdots$. For any $\mathcal J \subset \{0, \ 1, \ \dots\}$, we denote the orthogonal projection onto modes corresponding to $\mathcal J$ by $\Pi_{\mathcal J}\colon L^2(X_e) \to L^2(X_e)$, so that
\[
 (\Pi_{\mathcal J} u)(r,y) := \sum_{j \in \mathcal J} \phi_j(y)\int_Y u(r,y')  \phi_j (y') d\!\vol(y'),
\]
where $y$ and $y'$ denote points in $Y$. Then $\|\Pi_{\mathcal J} \chi_\Pi(r)\|_{L^2(X) \to L^2(X)} =1$, unless $\mathcal J$ is empty.

\begin{thm}\label{t:away}
Fix $s>1/2$ and $c_{\mathcal J} >0$. Let $\mathcal J := \{j \mid E_j:=E_0-h^2\sigma_j^2 \not\in[-c_{\mathcal J}h,c_{\mathcal J}]\}$.
Define a microlocal cutoff $\chi_{\mathcal J} \colon L^2(X) \to L^2(X)$ by putting
\begin{equation}\label{e:chidef}
 \chi_{\mathcal J}u :=
\begin{cases}
 \left(\Pi_{\mathcal J} \chi_\Pi(r) + \sqrt{V_L(r) + f(r)^{-4/(d-1)} - 1}\right)u, \quad &u \in L^2(X_e), \\
u, \quad &u \in L^2(X \setminus X_e),
\end{cases}
\end{equation}
and then extending to general $u \in L^2(X)$ by linearity.
There are $C>0$ and $h_1>0$ such that
\begin{equation}\label{e:taway}
\|(1 + r)^{-s}\chi_{\mathcal J}(P - E_0 - i \varepsilon)^{-1}(1+r)^{-s}\|_{L^2(X) \to L^2(X)}   \le C(1+a(h))h^{-1},
\end{equation}
for all $\varepsilon \in \mathbb R \setminus 0$ and for all $h \in (0,h_1]$. 
\end{thm}
By taking the adjoint, we see that \eqref{e:taway} implies
\begin{equation}\label{e:taway2}
\|(1 + r)^{-s}(P - E_0 - i \varepsilon)^{-1}\chi_{\mathcal J}(1+r)^{-s}\|_{L^2(X) \to L^2(X)}   \le C(1+a(h))h^{-1}.
\end{equation}

Note that the statement is strongest when $c_{\mathcal J}$ is chosen very small, much smaller than $E_0$. We think of $\chi_{\mathcal J}$ as cutting off away from (or, almost, projecting away from)
\[
T_\mathcal J = \{u \in L^2(X_e) \mid fu=u, \ V_L u = 0, \ \Pi_\mathcal Ju = 0 \} \subset L^2(X).
\]
Observe that the condition $E_j \in [-c_{\mathcal J}h,c_{\mathcal J}]$ corresponds, when $V_L=0$ and $f=1$, to the condition that $\rho^2 \in [-c_{\mathcal J}h,c_{\mathcal J}]$, where $\rho$ is the dual variable to $r$. In this sense $T_\mathcal J$ corresponds to a neighborhood of the bicharacteristics in $T^*X_e$ along which $r$ is constant, that is to say bicharacteristics trapped in the cylindrical ends. In this sense $\chi_{\mathcal J}$  cuts off away from the trapping in the cylindrical ends. The asymmetry in the interval $[-c_\mathcal J h, c_\mathcal J]$ is due to the fact that  our estimates are much easier when $E_j\le - Ch$ for any $C>0$ (see in particular the sentence following  \eqref{e:rechij} below); we do not expect this form of the interval to be optimal.

To simplify matters, in our discussion of the interpretation and context of this result we focus on the special case of the following Corollary, although most of the statements could be adapted to apply to the more general case.

\begin{cor}
Let $(X,g) = (\mathbb R^d, g)$ be as in Example \ref{ex:non1}. In the notation of that example, fix $\chi \in C_c^\infty(X)$ with $\supp \chi \subset \{(r,\theta) \in \mathbb R^d \mid r< R\}$, and fix $s>1/2$. Then there are $z_0>0$ and $C>0$ such that
\begin{equation}\label{e:cor}
 \|(1+r)^{-s} (-\Delta - z)^{-1} \chi \|_{L^2(X) \to L^2(X)} +  \|\chi (-\Delta - z)^{-1} (1+r)^{-s}\|_{L^2(X) \to L^2(X)} \le C/\sqrt{\re z},
\end{equation}
for all $z \in \mathbb C$ with $\re z \ge z_0$ and $\im z \ne 0$.
\end{cor}

Note that this $\chi$ is local, in contrast to the microlocal $\chi_{\mathcal J}$ of Theorem 3.2. 
Recall that $R$ is the radius at which the cylindrical end begins; hence $\chi$ is a cut off away from the trapping on the cylindrical end, and in this example there is no other trapping.  The right hand side of \eqref{e:cor} is the usual nontrapping upper bound, cf. \eqref{e:pdbig} and the bound of $Ch^{-1}$ in \eqref{e:pdbigh}. There have been many results in asymptotically Euclidean, conic, and hyperbolic scattering proving that such nontrapping bounds hold when one cuts off away from trapping on \textit{both} sides of the resolvent: these go back to work of Cardoso and Vodev \cite{cv}, refining an earlier result of Burq \cite{burq}. Intriguingly, in \eqref{e:cor} we get a nontrapping bound by applying a spatial cutoff  away from trapping on only \textit{one} side of the resolvent; to our knowledge no such result is known in asymptotically Euclidean, conic, and hyperbolic scattering, although a related weaker bound can be found in \cite{bz, chr, dv2} (and note that the weaker bound is shown to be optimal in a special example in \cite{dy}). A possible interpretation is the following: unlike in any of the examples studied in \cite{bz, dv2}, in Example \ref{ex:non1} the set $K$ of bicharacteristics  trapped as $t \to + \infty$ \textit{and}  $t \to -\infty$ is the same as the set $\Gamma_\pm$ of bicharacteristics   trapped as $t \to + \infty$ \textit{or} $t \to -\infty$, and one expects resolvent estimate losses due to mild trapping to be concentrated on $\Gamma_\pm$.

On the other hand, in \cite{ddz} it is shown that for a ``well in an island'' semiclassical Schr\"odinger operator (in which case incidentally  $K$ does equal $\Gamma_\pm$), losses due to trapping extend beyond $\Gamma_\pm$ and cutting off on one side only is not enough to give nontrapping bounds; as discussed in that paper, this is closely related to the fact that the trapping in this case is stable (so that tunneling can produce losses away from $\Gamma_\pm$), unlike in Example \ref{ex:non1} or in the examples in \cite{dv2}. It is then natural to ask: when is cutting off a resolvent away from trapping on one side sufficient to give nontrapping bounds, and when is it necessary to cut off on both sides?

\subsection{Examples with no trapping away from the ends}\label{s:nonex} Let $X$ have no boundary and let  $K_{E_0}$ be the set 
of bicharacteristics of $P$ at energy $E_0$ which do not intersect $T^*X_e$. If $K_{E_0} = \varnothing$, then it is essentially well-known that
\begin{equation}\label{e:wnon}
  \|(P -i W_K(r) -E_0 )^{-1}\|_{L^2(X) \to L^2(X)} \le C h^{-1};
\end{equation}
the proof of \eqref{e:wnon} follows from the proof of 
\cite[Theorem 6.11]{dz} or that of  \cite[Proposition 3.2]{d:cusp}. 
In the case that $|V| \le C h$, demanding that $K_{E_0} = \varnothing$ is equivalent to demanding that all maximally extended geodesics on $X$ intersect $X_e$; specific examples are given in Example \ref{ex:non1}.

\subsection{Hyperbolic and normally hyperbolic trapped sets.}\label{s:hypex} If $K_{E_0} \ne \varnothing$  we cannot hope to have \eqref{e:wnon}, but if $K_{E_0}$ is hyperbolic or normally hyperbolic then we may have
\begin{equation}\label{e:wlog}
  \|(P-i W_K(r) -E_0 )^{-1}\|_{L^2(X) \to L^2(X)} \le
C\log(h^{-1})h^{-1}.
\end{equation}
In the case of a closed hyperbolic orbit, such bounds are due to Burq \cite{burqs} and Christianson \cite{chr}. For  hyperbolic trapped sets satisfying a pressure condition they are due to Nonnenmacher and Zworski \cite{nz1}, and for normally hyperbolic trapped sets to Wunsch and Zworski \cite{wz} and to Nonnenmacher and Zworski \cite{nz2} (and see also \cite{dy}). Some recent surveys of the substantial wider literature concerning estimates like \eqref{e:wlog} can be found in \cite{nonn, z:sur}.

To deduce \eqref{e:wlog} from \cite{nz1} or  \cite{nz2}, note that the difference between \eqref{e:wlog} and \cite[(2.7)]{nz1} or \cite[(1.18)]{nz2} lies in the assumptions in the region where $W_K = 1$. But in this region $P-iW_K$ is semiclassically elliptic, so the discrepancy can be removed using a parametrix $G'$ analogous to the one in \eqref{e:g} below, and rather than having to go through a procedure like that in \S\ref{s:gluing} we just have $(P-i W_K(r) -E_0 )G' = I +  O(h^{\infty})$.

Rather than discussing the general dynamical assumptions further, we now specialize to more concrete examples.

Let $(X,g_H)$ be a conformally compact manifold of constant negative curvature.  We recall that this means that the metric $g_H$ is  asymptotically hyperbolic in the sense of \cite{mm} (see also \cite[\S 5.1]{dz}), so there is an open set $X_e'$ and $R \in \mathbb R$ such that $X \setminus X_e'$ is compact and
\[
X_e' = (R,\infty)_{r} \times Y, \qquad g_H|_{X_e'} = d{ r}^2 + e^{2 r}g_Y(e^{- r}),
\]
where $Y$ is a compact, not necessarily connected, manifold without boundary and $g_Y(x)$ is a family of metrics on $Y$ depending smoothly on $x$ up to $x = 0$. Such a `normal form' of the metric was first found in \cite{gl}, and it is also in \cite[\S5.1.1]{dz}.

We modify the metric to obtain a manifold with cylindrical ends in the following way. We first observe that, denoting points in $T^*X_e'$ by $(r,y, \rho, \eta)$, where $y \in Y$, $ \rho$ is dual to $r$, and $\eta$ is dual to $y$, along $g_H$-geodesics we have
\[
\frac{d^2}{dt^2} r =: \ddot  r =  -2 \partial_{r} (e^{-2r}|\eta|^2_{r,y}) = 4 e^{-2r}|\eta|^2_{r,y}(1+O(e^{-r})),
\]
where the length $|\eta|_{r,y}$ is taken with respect to the dual metric to $g_Y(e^{-r})$. 
Hence, after possibly redefining $R$ to be larger, we may suppose that $\ddot r \ge 2 e^{-2r}|\eta|^2_{r,y}$  for $r \ge R$, and in particular that no bounded $g_H$-geodesics intersect $\overline{X_e'}$. Indeed,  since $E_0:= \rho^2 + e^{-2r}|\eta|^2_{r,y}$ is conserved and $\dot r = 2 \rho$, in $X_e'$ we have
\[
 \ddot r \ge  2 e^{-2r}|\eta|^2_{r,y} =  2E_0 - \dot r^2/2,
\]
 which means $r$ is not bounded for all $t$.

 Fix $\chi_H \in C^\infty(\mathbb R; [0,1])$ such that $\chi_H(r) = 1$ near $(-\infty, R]$ and $\chi_H(r) = 0$ near $[R+1,\infty)$, and fix $ F\in C^\infty([R,\infty),(0,\infty))$ such that $F'$ is compactly supported, positive on the interior of its support, and such that $F'(r)>0$ for $r \le R+2$. Take $g$ such that $g|_{X \setminus X_e'} = g_H|_{X \setminus X_e'}$, and 
\[
g|_{X_e'} = \chi_H(r) g_H + C_g(1-\chi_H(r))\left(dr^2 +  F(r)g_Y(0)\right).
\]
We claim that if $C_g$ is large enough, then $\ddot r \ge 0$ along $g$-geodesics in $\overline{X_e}$. Indeed, 
\[
 \ddot r/2 =  - \chi_H(r) \partial_r (e^{-2r}|\eta|^2_{r,y}) + C_g(1-\chi_H(r))F'(r)|\eta|^2_0  - \chi'_H(r)(e^{-2r}|\eta|^2_{r,y} - C_g F(r)|\eta|^2_0 ),
\]
so it is enough to take $C_g$ large enough that on $T^*\supp \chi'_H(r)$ we have $e^{-2r}|\eta|^2_{r,y} \le C_g F(r)|\eta|^2_0$.

Now we may take $X_e$ to be the part of $X_e'$ in which $r > R+1$, and, after redefining $r$ by \eqref{e:rmapsto}, we see that it remains only to check \eqref{e:w}.

We take $W_K \in C^\infty(\mathbb R; [0,1])$ which is $1$ near $[R+2,\infty)$ and $0$ near $(-\infty,R+1]$, and suppose $|V| \le Ch$ and $E_0=1$. Let $K$ denote the set of trapped unit speed geodesics of $(X,g_H)$, regarded as a subset of $T^*X$. We see that $K$ is also the set of the bicharacteristics of $P$ at energy $E_0$ which do not intersect $T^*X_e$, and that $g_H =g$ near the projection of $K$ onto $X$.

Let $d_K$ be the Hausdorff dimension of $K$. If $d_K<d$, then the assumptions of \cite{nz1} are satisfied, and  \eqref{e:wlog} holds. 

If $d=2$ and $V \equiv 0$, then we can dispense with the requirement that $d_K<d$ thanks to a recent result of Bourgain and Dyatlov \cite[Theorem 2]{bd} (this is the case presented in Example \ref{ex:3fun1} above). To do this we use the fact (see \cite[Lemma 4.7]{burqs} or e.g. \cite[Proof of (6.3.10)]{dz}) that \cite[(1.1)]{bd} implies
\[
 \|\chi(-h^2\Delta_0 - E_0 -i0)^{-1}\chi\|_{L^2(X) \to L^2(X)} \le C \log(h^{-1})h^{-1},
\]
for any $\chi \in C_c^\infty(X)$. Then the gluing result of \cite[Theorem 2.1]{dv} together with the semiclassically outgoing property of $(-h^2\Delta_0 - E_0 -i0)^{-1}$ (established by Vasy in \cite{vasy}  and see also \cite[Theorem 5.34]{dz}) implies \eqref{e:wlog}. In the interest of brevity we do 
not discuss this further here.

\subsection{Warped products with embedded eigenvalues}\label{s:warp}

Let $X := \mathbb R \times Y$ and $g:=dr^2 + f(r)^{4/(d-1)}g_Y$ for some $f \in C^\infty(\mathbb R;(0,1])$ which is 1 on $\mathbb R \setminus (-R,R)$ for some $R>0$ and has a nondegenerate minimum as its only critical point in $(-R,R)$: see Figure \ref{f:hourglass}.

\begin{figure}[h]
\vspace{.3 cm}
\includegraphics[width=10cm]{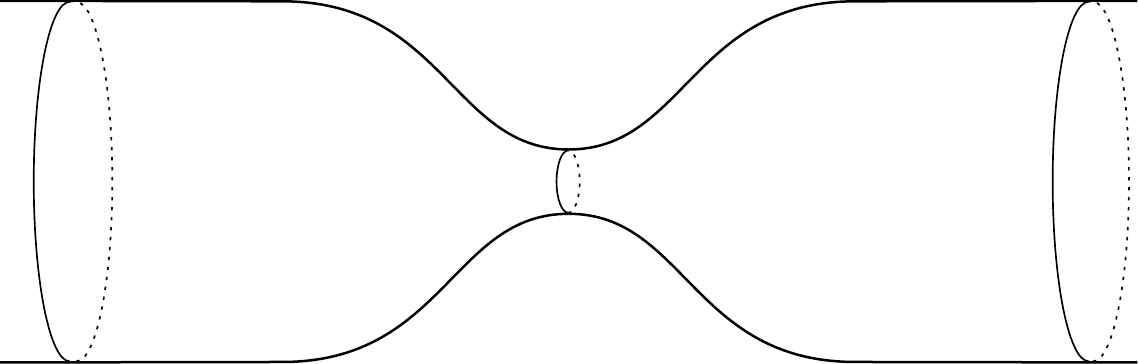}
 \caption{An hourglass shaped surface of revolution.}\label{f:hourglass}
\end{figure}

Suppose $V = h^2 V_W$, with $V_W = V_W(r)\in C_c^\infty((-R,R))$. Then the part of the trapped set away from the cylindrical ends  is normally hyperbolic and we have \eqref{e:wlog} (see \cite[(6.3.10)]{dz}, and see also \cite{cw, c} for the case of a degenerate minumum where incidentally  we also have \eqref{e:w}). Consequently, by Theorem \ref{t:cont}, there is $z_0>0$ such that for all  $s_1, \ s_2 >1/2$ such that $s_1 + s_2>2$, there is $C>0$  such that
\[
\|(1 + |r|)^{-s_1}(- \Delta + V_W  - z)^{-1}(1+|r|)^{-s_2}\|_{L^2(X) \to L^2(X)} \le C,
\]
for all $z \in \mathbb C$ with $\re z \ge z_0$ and $\im z \ne 0$. In particular the spectrum of $-\Delta + V_W$ is absolutely continuous on $(z_0,\infty)$.

But if $f$ and $V_W$ are suitably chosen, then $ \Delta + V_W$ has an eigenvalue embedded in the spectrum in $[0,z_0]$. Indeed, we have
\[
 \Delta = f(r)^{-1} \left(\sum_{j=0}^\infty\left(\partial_r^2 - 
f''(r)f(r)^{-1} - \sigma^2_j f(r)^{-4/(d-1)}\right) \phi \otimes \phi \right) f(r),
\]
where  $\{ \phi_j\}_{j=0}^\infty$ is a complete set of real-valued orthonormal eigenfunctions of the Laplacian on $Y$ and $-\Delta_Y \phi_j=\sigma_j^2 \phi_j$.
 For $J \in \mathbb N_0$,  consider the effective potential 
\[V_J(r): = f''(r)f(r)^{-1} + \sigma^2_J (f(r)^{-4/(d-1)} - 1) + V_W(r).\]
Then $D_r^2 + V_J$ has an eigenvalue as long as $\int V_J(r) dr \le 0$  by \cite[Theorem XIII.110]{rs}, and this corresponds to an embedded eigenvalue for $- \Delta + V_W$ as long as it is positive, for which it suffices to have $\min V_J(r) > - \sigma^2_J$. For example, we may take $f$ such that   $\int(f(r)^{-4/(d-1)} - 1) \le 1/4$ and $V_W \in C_c^\infty((-R,R);[- \sigma_J^2/2,0])$ such that  $V_W(r) = - \sigma_J^2/2$ on $[-R/2,R/2]$, and then $J$ sufficiently large.

By elaborating the above constuction one can also find examples with any finite number of embedded eigenvalues.

It is not clear whether there are examples of manifolds with cylindrical ends such that $-\Delta$ has a finite but nonzero number of eigenvalues. For all known examples where eigenvalues occur, the existence of infinitely many eigenvalues is either also established \cite{cz, parn} or at the least it is not ruled out \cite{kk}. On the other hand $0$ is always a resonance of $-\Delta$ on a manifold with cylindrical ends, with the constant functions as resonant states, unless there is a boundary condition somewhere that eliminates them.

\section{Proof of Theorems \ref{t:cont} and \ref{t:away}}\label{s:proof}

\subsection{Outline of proof}\label{s:outline}

The idea of the proofs is to define a parametrix for $P - z$ by
\begin{equation}\label{e:g}
G:= \chi_K(r-1)(P-iW_K(r)-z)^{-1}\chi_K(r) + \chi_e(r+1)(P_e - z)^{-1} \chi_e(r),
\end{equation}
where  $\chi_e, \chi_K \in C^\infty(\mathbb R)$ obey $\chi_e + \chi_K = 1$, 
$\supp \chi_e \subset (3,\infty)$, and $\supp \chi_K \subset (-\infty,4)$, and $P_e$ is a suitably chosen differential operator such that $P_e = P$ on the part of $X$ where $r>2$. Then
\[
 (P-z)G = I +  [h^2 D_r^2, \chi_K(r-1)](P-iW_K(r)-z)^{-1}\chi_K(r) +  [h^2 D_r^2, \chi_e(r+1)](P_e - z)^{-1} \chi_e(r),
\]
and we will construct an inverse for $(P-z)$ by removing this remainder using a Neumann series; although the remainder above need not be small, we will see that powers of it are. We call the part of $X$ where $r \in (2,5)$ the \textit{resolvent gluing region}, because the functions in the range of the remainder are supported in that region. To prove that powers of the remainder are small, we will need to know that:
\begin{enumerate}
 \item The resolvents  of $P-iW_K(r)$ and $P_e $ obey estimates analogous to \eqref{e:2n-1} and \eqref{e:taway}. This is the case for $P-iW_K(r)$ thanks to the assumption \eqref{e:w}, and we will prove it for a suitable choice of $P_e$ in \S\ref{s:jbig} and \S\ref{s:jsmall}.
\item The resolvents of $P-iW_K(r)$ and $P_e $ obey improved estimates when multiplied by cutoffs with suitable support properties in the resolvent gluing region, corresponding to a (special case of a) \textit{semiclassically outgoing condition} so that we are able to remove the remainders. The needed estimates are proved in \cite{dv} for  $P-iW_K(r)$  and in \S\ref{s:jbig} and \S\ref{s:jsmall} for $P_e$.
\end{enumerate}

We combine these estimates to prove  Theorems \ref{t:cont} and \ref{t:away} in \S\ref{s:gluing}. There we follow a procedure analogous to that in \cite{dv}, but with some finer analysis of remainders to remove the losses due to trapping in the cylindrical end (see also \cite[\S 3]{d:cusp} for another, in some ways related, variation on this resolvent gluing procedure).

\subsection{Model operators for $X_e$}\label{s:xc}

On $X_e$, $\Delta$ can be written as a direct sum of one-dimensional 
Schr\"odinger operators:
\[
 \Delta|_{X_e} = f(r)^{-1} \left(\sum_{j=0}^\infty\left(\partial_r^2 - 
f''(r)f(r)^{-1} - \sigma^2_j f(r)^{-4/(d-1)}\right) \phi \otimes \phi \right) f(r),
\]
where  $\{ \phi_j\}_{j=0}^\infty$ is a complete set of real-valued orthonormal eigenfunctions of the Laplacian on $Y$ and $-\Delta_Y \phi_j=\sigma_j^2 \phi_j$. We will introduce model operators $P_j$ obeying
\begin{equation}\label{e:pccagree}
 P_j|_{[2,\infty)}=  -h^2 \partial_r^2 + V_j(r), \qquad V_j(r) := V(r) +  h^2 f''(r)f(r)^{-1} + h^2 \sigma_j^2 (f(r)^{-4/(d-1)} - 
1),
\end{equation}
 and we will be studying them near the energy levels 
\[
E_j:=E_0 - h^2 \sigma_j^2.
\]

We will study 
two ranges of $j$ separately, and the model operators $P_j$ will act on 
different spaces depending on $j$. These two ranges correspond to different behavior in the resolvent gluing region, which is the part of $X$ where $r \in (2,5)$ (see \S\ref{s:outline}). To define the ranges, fix $E_*\in \mathbb R$, independent of $h$, such that
\[
 0<E_* \le c_{\mathcal J},
\]
where $c_{\mathcal J}$ is as in the statement of Theorem \ref{t:away}, 
and
\begin{equation}\label{e:estarimp}
 E_j \le E_* \Longrightarrow h^2\sigma_j^2f(5)^{-4/(d-1)} \ge E_0;
\end{equation}
note that the conditions are compatible because $E_j = 0$ when $E_0 = h^2 \sigma_j^2$ and $f(5)<1$.

The first range we consider is $E_j \le E_*$; in this range the set where $r<5$ is classically forbidden because $V_j>E_j$, and we control remainders in the gluing region using Agmon estimates, taking care to prove that our estimates are uniform as $j \to \infty$ (although the effective potentials $V_j$ become unbounded as $j \to \infty$, they are nonnegative, so the relevant estimates actually get better in this limit). The second range is $E_j \ge E_*$; in this range the set where $r<5$ is not classically forbidden, but the energy levels $E_j$ are bounded below by a positive constant and the effective potentials $V_j$ are repulsive, so nontrapping propagation of singularities estimates hold, which we can use to control the remainders in the gluing region (once again we take care to prove that the estimates are uniform in $j$). 

For the first range of $j$ we define the operators $P_j$ to act on $L^2(\mathbb R_+)$, with a Dirichlet boundary condition at $0$, in order to be able to use Theorem \ref{t:half} (the Dirichlet boundary condition makes it easier to analyze the behavior of the resolvent when $|E_j|$ is small). For the second range of $j$ it is more convenient to work over $\mathbb R$ than $\mathbb R_+$, in order to avoid reflection phenomena when studying propagation of singularities.

\subsection{Analysis when $E_j \le E_*$}\label{s:jbig}
In \S\ref{s:jbig} all function norms and inner products are in $L^2(\mathbb R_+)$, and operator norms are $L^2(\mathbb R_+) \to L^2(\mathbb R_+)$, unless otherwise specified.

For this range of $j$, we  put 
\begin{equation}\label{e:pjbigdef}
 P_j:= h^2 D_r^2 + V_j(r),
\end{equation}
regarded as a self-adjoint operator on $L^2(\mathbb R_+)$  with a Dirichlet 
boundary condition at $r=0$. 

We first prove resolvent estimates for $P_j$ analogous to \eqref{e:2n-1} and \eqref{e:taway}.

\begin{prop}\label{p:resestjbig}
Fix $s_1, \ s_2, \ s >1/2$ such that $s_1 + s_2>2$. Then
\begin{equation}\label{e:resestjbig}
 \| (1+r)^{-s_1}(P_j - E_j - i \varepsilon)^{-1} (1+r)^{-s_2} \| \le C h^{-2},
\end{equation}
and
\begin{equation}\label{e:resestjbigchi}
 \|(1 + r)^{-s}\chi(r)(P_j - E_j - i \varepsilon)^{-1}(1+r)^{-s}\| + \|(1 + r)^{-s}(P_j - E_j - i \varepsilon)^{-1}\chi(r) (1+r)^{-s}\|  \le Ch^{-1},
\end{equation}
for all $\varepsilon \in \mathbb R \setminus 0$, $j \in \mathbb N$ such that $E_j \le E_*$, where
\[
 \chi(r) = \sqrt{V_L(r) + f(r)^{-4/(d-1)} - 1}.
\]
\end{prop}

\begin{proof} 

The idea of the proof is to apply Theorem \ref{t:half}; more precisely \eqref{e:resestjbig} corresponds to \eqref{e:pdsmall1} (see also \eqref{e:pdsmallh}), and \eqref{e:resestjbigchi} corresponds to \eqref{e:pdvbound} (see also \eqref{e:pdvboundh}). 

Before beginning the proof proper, by way of outline let us briefly discuss the terms in $V_j$, and explain how they each do or do not satisfy \eqref{e:vmhyp}. The term $h^2\sigma_j^2(f(r)^{-4/(d-1)} - 1)$ does satisfy it thanks to \eqref{e:fdelta0} and \eqref{e:estarimp}, and moreover those bounds and $f(r)<1$ for $r<6$ imply that the term is nontrivial when $r<6$. The term $V_L$ satisfies it, and we think of it as being harmless. The terms $V_S$ does not satisfy it, but we will show that its effect is compensated by that of the $h^2\sigma_j^2(f(r)^{-4/(d-1)} - 1)$ term. The most difficult term to treat is the $h^2 f''(r)f(r)^{-1}$ term. This term may prevent $h^{-2}V_j$ from satisfying \eqref{e:vmhyp}, but we will show that thanks to \eqref{e:estarimp} we can treat it as a small perturbation.

More precisely, let
\[
 V_M(r):= V_j(r) - h^2 f''(r)f(r)^{-1},
\]
and observe that for $h$ sufficiently small $V_M$ obeys \eqref{e:vmhyp} for some $\delta_V>0$, since $V_L$ and $f^{-4/(d-1)} -1$ obey it and  $|V_S| + |V_S'| \le C (f^{-4/(d-1)} - 1)$ thanks to \eqref{e:estarimp}. Indeed, to see that $f^{-4/(d-1)} -1$ obeys it we write, using $\alpha:=4/(d-1)$ and \eqref{e:fdelta0},
\[
- (f(r)^{-\alpha} - 1)' = \alpha f'(r)f(r)^{-\alpha -1} \ge \alpha \delta_0 \frac{f(r)^{-\alpha-1}- f(r)^{-\alpha}}{1+r}\ge \frac{f(r)^{-\alpha}-1}{C(1+r)},
\]
where we also used the fact that if $a<b$ then
\begin{equation}\label{e:fab}
 C^{-1}(1-f)\le f^a- f^b \le C(1-f).
\end{equation}

 Hence by \eqref{e:pdsmall1} with $V_D = h^{-2}V_M$, we have
\begin{equation}\label{e:vmh2}
  \| (1+r)^{-s_1}(h^2D_r^2 + V_M - E_j - i \varepsilon)^{-1} (1+r)^{-s_2} \| \le C h^{-2}.
\end{equation}
Note that by the resolvent identity
\begin{equation}\label{e:resid}\begin{split}
 (1+r)^{-s_1} (P_j -  E_j& -  i  \varepsilon)^{-1} (1+r)^{-s_2} = (1+r)^{-s_1} (h^2D_r^2 + V_M - E_j - i 
\varepsilon)^{-1} (1+r)^{-s_2}  \\ &\times \sum_{k=0}^\infty \left[-(1+r)^{s_2}h^2 f''(r)f(r)^{-1} (h^2D_r^2 + V_M  - E_j- i 
\varepsilon)^{-1}(1+r)^{-s_2}\right]^k,
\end{split}\end{equation}
the proof of \eqref{e:resestjbig} is reduced to the proof of
\begin{equation}\label{e:vm12}
\|(1+r)^{s_2}h^2 f''(r)f(r)^{-1} (h^2D_r^2 + V_M  - E_j- i 
\varepsilon)^{-1}(1+r)^{-s_2}\| \le 1/2.
\end{equation}
But by \eqref{e:pdvbound}, with $\theta = 1$ and $V_D = h^{-2}V_M \ge h^{-2}(f^{-4/(d-1)} - 1)/C$ (again using \eqref{e:estarimp}), we have
\[
\|(f(r)^{-4/(d-1)} - 1)^{\frac 12} (1+r)^{-\frac 12}(h^2D_r^2 + V_M  - E_j- i 
\varepsilon)^{-1}(1+r)^{-s_2}\| \le C h^{-1},
\]
and interpolating this with \eqref{e:vmh2} gives
\[
 \|(f(r)^{-4/(d-1)} - 1)^{\frac 14} (1+r)^{-\frac{s_1}2 - \frac 14}(h^2D_r^2 + V_M  - E_j- i 
\varepsilon)^{-1}(1+r)^{-s_2}\| \le C h^{-3/2}.
\]
 Hence to prove \eqref{e:vm12}, and consequently also \eqref{e:resestjbig}, it is enough to show that
\begin{equation}\label{e:f''f}
 (1+r)^{s_2} |f''(r)| \le C  (f(r)^{-4/(d-1)} - 1)^{\frac 14} (1+r)^{-\frac{s_1}2 - \frac 14}.
\end{equation}
To prove \eqref{e:f''f} we will use the fact that any bounded $\varphi \in C^2([r,\infty);[0,\infty))$ satisfies
\begin{equation}\label{e:mvt}
|\varphi'(r)|^2 \le 2 \sup\varphi \sup |\varphi''|,
\end{equation}
where the suprema are  taken over $[r,\infty)$.
Indeed, by Taylor's theorem, for every $t \ge 0$ there is $t_0 \in [r,r+t]$ such that
\[
t|\varphi'(r)| = |\varphi(r+t) - \varphi(r) - t^2  \varphi''(t_0)/2| \le \sup \varphi + t^2\sup |\varphi''|/2,
\]
and taking $t =  |\varphi'(r)|/ \sup|\varphi''|$  gives \eqref{e:mvt}. Applying \eqref{e:mvt} once with $\varphi = f'$ and once with $\varphi = 1 - f$ gives
\[
 |f''(r)|^4 \le 4 \sup |f'|^2 \sup |f'''|^2 \le 8 \sup(1 - f) \sup |f''| \sup |f'''|^2 = 8 (1 - f(r)) \sup |f''| \sup |f'''|^2,
\]
where the suprema are still all taken over $[r,\infty)$. Applying \eqref{e:fdecay} gives
\[
 |f''(r)| \le C (1 - f(r))^{\frac 14}(1+r)^{-2 - \frac {3 \delta_0}4}.
\]
By \eqref{e:fab} this implies \eqref{e:f''f} as long as $s_1 + 2 s_2 \le (7 + 3 \delta_0)/2$, which we may suppose without loss of generality. This completes the proof of \eqref{e:resestjbig}.

The proof of \eqref{e:resestjbigchi} proceeds along similar lines. Applying \eqref{e:resid}  with $s_1 = s_2 = s$ allows us to reduce the proof of the bound on the first term in \eqref{e:resestjbigchi} to the proof of
\begin{equation}\label{e:resestjbigchi1}
 \|(1 + r)^{-s}\chi(r)(h^2D_r^2 + V_M - E_j - i \varepsilon)^{-1}(1+r)^{-s}\|  \le Ch^{-1}.
\end{equation}
But \eqref{e:resestjbigchi1} follows from \eqref{e:pdvbound} with $\theta = 1$ and $V_D = h^{-2}V_M \ge h^{-2}(V_L + f^{-4/(d-1)} - 1)/C = h^{-2} \chi^2/C$. The bound on the second term of \eqref{e:resestjbigchi} follows from the bound on the first term after taking the adjoint.
\end{proof}

We will 
also need the following Agmon estimates:
 
 \begin{prop}
Let $R \in (0,5]$,  $\chi_- \in C_c^\infty((0,R)), \ \chi_+ \in C_c^\infty((R,\infty)),$ and $s>1/2$. Then
\begin{equation}\label{e:agmonoverlap}
 \| \chi_-  (P_j - E_j - i \varepsilon)^{-1} (1+r)^{-s}\|_{L^2(\mathbb R_+) \to H^1_h(\mathbb R_+)} +  \| (1+r)^{-s}  (P_j - E_j - i \varepsilon)^{-1} \chi_- \| \le C,
\end{equation}
\begin{equation}\label{e:agmondisjoint}
 \| \chi_- (P_j - E_j - i \varepsilon)^{-1} \chi_+ \|
\le e^{-1/(Ch)},
\end{equation}
for all $\varepsilon \in \mathbb R \setminus 0$, and $j \in \mathbb N$ such that $E_j \le E_*$.
\end{prop}
 Recall that the norms without subscripts are $L^2(\mathbb R_+) \to L^2(\mathbb R_+)$ here, and that $\chi_-$ is supported in the classically forbidden region for $P_j-E_j$.

 \begin{proof}
These are similar to the usual Agmon estimates as in  \cite[\S 7.1]{z} but we keep track 
of the $j$ dependence. 

Let $v \in L^2(\mathbb R_+)$, and let $u := (P_j - E_j - i 
\varepsilon)^{-1}(1+r)^{-s}v$. Fix $\varphi_0 \in C_c^\infty((0,R);[0,1])$ which is 
identically 1 on a neighborhood $I$ of $\supp \chi_-$, and let $\varphi(r) := 
m\varphi_0(r)$, for  a constant $m$ to be chosen 
later. Then define
\[\begin{split}
P_\varphi :&= e^{\varphi/h}(P_j- E_j - i \varepsilon)e^{-\varphi/h} \\
&= h^2D_r^2 + 2 i  \varphi' h D_r + V_j - \varphi'^2 + h 
\varphi'' - E_j - i \varepsilon.
\end{split}\]
Put $w := \chi_0 e^{\varphi/h}u$, where $\chi_0 \in C_c^\infty((0,R))$ is 1 near 
$\supp \varphi$. Using $ \re \langle 2 h \varphi' w', w \rangle = - h \langle 
\varphi'' w, w\rangle$,  write
\[
\re \langle P_\varphi w, w \rangle = \|hw'\|^2 + \langle (V_j - \varphi'^2  - 
E_j) w, w \rangle.
\]
We now observe that, using \eqref{e:estarimp} and the fact that $1 - f(r)^{-4/(d-1)}> 1 - f(5)^{-4/(d-1)} > 0$ for $r \in (0,5)$,
we can choose $m>0$ small enough, independent of $h$ 
and $j$, such that there is $c_0>0$ independent of $h$ and $j$ for 
which $V_j - \varphi'^2  - E_j>c_0$ on $\supp w$ for $h$ small enough. This 
implies
\[
\|w\| \le C\| P_\varphi w\| \le C \|e^{\varphi/h}\chi_0 v\|+ C \|[P, \chi_0]u\|,
\]
where we used $\varphi \chi_0'  = 0$ to deduce $ [P_\varphi,\chi_0]   e^{\varphi/h}u
= [P, \chi_0]u$. We use an elliptic estimate to bound the commutator term: for $\chi_1 \in C_c^\infty((0,R))$ we have, 
using $V_j - E_j \ge V_0 - E_0 \ge - C$,
\begin{equation}\label{e:hu'}\begin{split}
C \|\chi_1 v\|\|\chi_1u\| &\ge \re \langle  (1+r)^{-s} v, \chi_1^2u\rangle = \re \langle (P_j - E_j) u, \chi_1^2 u \rangle 
\\ &\ge \|\chi_1 hu'\|^2 - C h\|\chi_1 h u' u\|_{L^1(\mathbb R_+)} - C\|\chi_1 u\|^2,
\end{split}\end{equation}
from which it follows that, provided $\chi_2 = 1$ near $\supp \chi_0$,
\[
\|[P, \chi_0]u\| \le Ch \|\chi_2 u\| + Ch \|\chi_2 v\| \le C h^{-1}\|v\|,
\]
where we used \eqref{e:resestjbig}. Consequently
\begin{equation}\label{e:aginti}
\int_I |u|^2 = e^{-2m/h}\int_I |w|^2 \le C e^{-2m/h}\left( \|e^{\varphi/h}\chi_0 v\|^2 +  h^{-2}\|v\|^2 \right) \le C\|v\|^2,
\end{equation}
where we used $\varphi \le m$.

To estimate $u'$ we apply \eqref{e:hu'} with $\chi_1 \in 
C_c^\infty(I)$, giving
\[
\|\chi_1 hu'\|^2  \le C \left(\int_I |u|^2dr + \|\chi_1 hv\|^2\right),
\]
which implies the bound on the first term of \eqref{e:agmonoverlap}. The bound on the second term follows from taking the adjoint, and \eqref{e:agmondisjoint} follows from the fact that if $\supp v \subset (R,\infty)$, then $\chi_0 v = 0$ and we can improve \eqref{e:aginti} to
\[
 \int_I |u|^2 = e^{-2m/h}\int_I |w|^2 \le C e^{-2m/h} h^{-2}\|v\|^2.
\]
 \end{proof}

\subsection{Analysis when $E_j > E_*$}\label{s:jsmall}
In \S\ref{s:jsmall} all function norms and inner products are in $L^2(\mathbb R)$, and operator norms are $L^2(\mathbb R) \to L^2(\mathbb R)$, unless otherwise specified. 

For this range of $j$ the Agmon estimate \eqref{e:agmondisjoint} must be replaced by a propagation of 
singularities estimate. It is convenient to introduce a complex absorbing 
barrier and to work over $\mathbb R$: let $W_e \in C^\infty(\mathbb R;[0,1])$ be 
$1$ near $(-\infty,1]$ and $0$ near $[2,\infty)$, and let 
\[V_{j,0} := \chi_0 V_j,\]
 where $\chi_0 \in C^\infty(\mathbb R;[0,1])$ is $0$ near 
$(-\infty,0]$ and $1$ near $[1,\infty)$. We now put
\[
 P_j:= h^2 D_r^2 + V_{j,0}(r) - i W_e(r),
\]
regarded as an unbounded operator on $L^2(\mathbb R)$ with domain $H^2(\mathbb R)$. 
We will prove
\begin{prop}
For any $s>1/2$ we have
\begin{equation}\label{e:resestjsmallh22}
 \|(1+r_+)^{-s} (P_j - E_j - i \varepsilon)^{-1}(1+r_+)^{-s}\| \le Ch^{-1},
\end{equation}
where $r_+ := \max\{0,r\}$. For any
 $\chi_- \in C_c^\infty((0,3)), \ \chi_+ \in C_c^\infty((3,\infty)), 
\ \psi \in C_c^\infty((0,\infty))$, we have
\begin{equation}\label{e:propsing2}
\| \chi_-(r) (P_j - E_j - i \varepsilon)^{-1}   \chi_+(r) \psi(hD_r) \| = O(h^\infty).
\end{equation}
Both \eqref{e:resestjsmallh22} and \eqref{e:propsing2} hold uniformly for all $\varepsilon > 0$, and for all $j \in \mathbb N_0$ such that $E_j > E_*$.
\end{prop}

Note that since $E_j$ is bounded from below away from 0, we can think of \eqref{e:resestjsmallh22} as the analogue of \eqref{e:pdbig} in this setting; we do not need a weight for $r <0$ because the $-iW_e$ term makes the operator $P_j - E_j - i \varepsilon$  semiclassically elliptic there. It is also similar to the usual nontrapping resolvent estimate as in \cite{vz} and in other papers cited therein, but we need an estimate which is uniform in $j$.

The propagation of singularities estimate \eqref{e:propsing2} is a microlocalized version of \eqref{e:resestjsmallh22}. The improved bound is due to the fact that solutions to the classical equations of motion $\dot r = 2 \rho$, $\dot \rho = - V'_j(r)$ with $r(0)>3$ and $\rho(0)>0$ cannot have $r(t)<3$ for any $t >0$.

\begin{proof}[Proof of \eqref{e:resestjsmallh22}]
 We prove \eqref{e:resestjsmallh22} using a microlocal positive commutator argument, rather than (as is probably possible) integration by parts arguments as in the proof of \eqref{e:pdbig}. We do this because the proof of \eqref{e:propsing2} follows along very similar lines, and the latter estimate does not seem to be provable by integration by parts arguments. The idea is to construct a microlocal commutant, based on the $w(r)\partial_r$ of the proof of \eqref{e:pdbig}, but which is nonnegative. This will be obtained as the quantization of an escape function, defined in \eqref{e:qdef} below. 

 As in \cite{vz} we will use the semiclassical scattering calculus, and we begin by recalling its relevant properties. We use $(r,\rho)$ to denote points in $T^*\mathbb R$, and for $l, \ m \in \mathbb R$ we define the symbol class $S^m_l$ to be the set of $a \in C^\infty(T^*\mathbb R)$ such that, for any $n_1, \ n_2 \in \mathbb N_0$ there is $C_{n_1, n_2}$ such that
\begin{equation}\label{e:scdef}
 |\partial^{n_1}_r\partial^{n_2}_\rho a(r,\rho)| \le C_{n_1, n_2}(1 + |r|)^{l-n_1}(1 + |\rho|)^{m- n_2},
\end{equation}
for all $(r,\rho) \in T^*\mathbb R$. We also write $S^\infty_l := \bigcup_m S^m_l$, $S^{-\infty}_l := \bigcap_m S^m_l$, and similarly for $S^m_\infty$ and $S^m_{-\infty}$.
 Below we will consider symbols depending on $h$ and $j$, and the constants $C_{n_1, n_2}$ in \eqref{e:scdef} will always be uniform with respect to those parameters.
For such $a$, we denote the semiclassical quantization by $\Op_h(a)$, which we define by 
\begin{equation}\label{e:quantdef}
\Op_h(a)u:= \frac 1 {2\pi  h} \int \!\!\! \int e^{i(r - r')\rho/h} a(r,\rho) 
u(r')dr'd\rho.
\end{equation}
When a symbol is denoted by a lowercase letter (with possible subscripts and superscripts), we will denote its quantization by the corresponding uppercase letter (with the same subscripts and superscripts, if any).

We recall the composition and adjoint formulas. If  $a \in S^{m_1}_{l_1}$ and $b \in S^{m_2}_{l_2}$, then there is $a \# b\in S^{m_1+m_2}_{l_1+l_2}$ such that
\[
 AB = \Op_h(a \# b),
\]
and, for any $N \in \mathbb N$,
\begin{equation}\label{e:compform}\begin{split}
 a \# b (r,\rho) &= e^{-ih\partial_{r'}\partial_{\rho'}}\left(a(r,\rho')b(r',\rho)\right)\Big|_{(r,\rho) = (r',\rho')} \\&= \sum_{k=0}^{N-1} \frac {(-ih)^k}{k!} \partial_\rho^k a(r,\rho)\partial_r^kb(r,\rho) + h^Nz_N(r,\rho),
\end{split}\end{equation}
where $z_N \in S^{m_1+m_2 -N}_{l_1+l_2 - N}$ is given by
\begin{equation}\label{e:comprem}\begin{split}
z_N(r,\rho):= \frac {(-i)^N} {(N-1)!}\int_0^1  (1-t)^{N-1}e^{-ith\partial_{r'}\partial_{\rho'}}\left(\partial_{\rho'}^Na(r,\rho') \partial_{r'}^N b(r', \rho)\right)\Big|_{(r,\rho) = (r',\rho')}   dt.
\end{split}\end{equation}
Indeed, \cite[Theorem 4.14]{z} gives the formula for Schwartz symbols, and \cite[Theorems 4.13 and 4.18]{z} give it for a larger class of symbols than the ones we consider, but with weaker bounds on $z_N$. The statement that $z_N \in S^{m_1+m_2 -N}_{l_1+l_2 - N}$ follows from applying \cite[Theorem 4.17]{z} to \eqref{e:comprem}. See also \cite[Proposition E.8]{dz}, \cite[(3) and (9)]{par}, \cite{shub}, and \cite[\S18.5]{hor} for similar expansions, and \cite{hs} for a much more general version.

Similarly, if $a \in S^{m}_{l}$ there is $a^* \in S^{m}_{l}$  such that the formal adjoint of $A$ is given by
\[
 A^* = \Op_h(a^*),
\]
and, for any $N \in \mathbb N$,
\begin{equation}\label{e:adform}
a^*(r,\rho) = e^{-ih\partial_{r}\partial_{\rho}} \bar a(r,\rho) =  \sum_{k=0}^{N-1} \frac {(-ih)^k}{k!} \partial_r^k \partial_\rho^k \bar a(r,\rho) + h^Nz_N(r,\rho),
\end{equation}
where this time $z_N \in S^{m -N}_{l - N}$ is given by
\[
z_N(r,\rho):= \frac {(-i)^N} {(N-1)!}\int_0^1 (1-t)^{N-1}e^{-ith\partial_{r}\partial_{\rho}}\partial_{r}^N \partial_{\rho}^N \bar a(r,\rho) dt.
\]

Let
\[
 p_j:= \rho^2 + V_{j,0}(r) - i W_e(r)
\]
be the semiclassical symbol of $P_j$ (in the sense that $p_j \in S^2_0$ and $P_j = \Op_h(p_j)$), let
\[
 R_j:= \inf \{r>0 \mid \textrm{both } V_j(r) = V_{j,0}(r) \textrm{ and } V_{j,0}(r)\le E_*/2\},
\]
so that $R_0 \le R_1  \le \cdots$, and let
\[
F_j := \{(r,\rho)\mid r \ge 1 \textrm{ and } \rho^2 \le 2E_0\} \setminus \{(r,\rho) 
\mid R_j < r \textrm{ and } \rho^2 < E_*/3\}.
\]
Note that each $F_j$ is a closed neighborhood of the energy surface $p_j = E_j$, and they have been chosen such that they form a nested sequence $F_0 \subset F_1 \subset \cdots$. Moreover, since we only consider $j$ such that $E_j > E_*$, all of the $F_j$ agree outside of a compact set: see Figure \ref{f:fj}.

\begin{figure}[h]
\labellist
\small
\pinlabel $r$ [l] at 239 87
\pinlabel $r$ [l] at 596 87
\pinlabel $\rho$ [l] at 17 195
\pinlabel $\rho$ [l] at 374 195
\pinlabel $R_j$ [l] at 96 87
\pinlabel $1$ [l] at 24 87
\pinlabel $\sqrt{E_*/3}$ [l] at -50 120
\pinlabel $\sqrt{2E_0}$ [l] at -35 180
\pinlabel $\textrm{the case }R_j>1$ [l] at 74 -15
\pinlabel $\textrm{the case }R_j<1$ [l] at 431 -15
\endlabellist
\includegraphics[width=6cm]{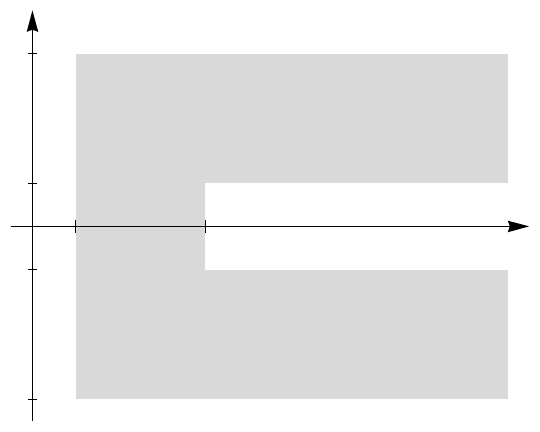}
\hspace{2cm}
\includegraphics[width=6cm]{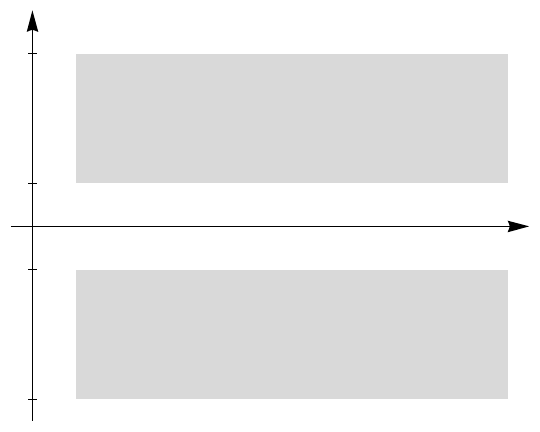}
\vspace{.5cm}
 \caption{The shaded regions are the sets $F_j$. They are closed nested neighborhoods of the energy surfaces $p_j = E_j$ which agree outside of a compact set.}\label{f:fj}
\end{figure}
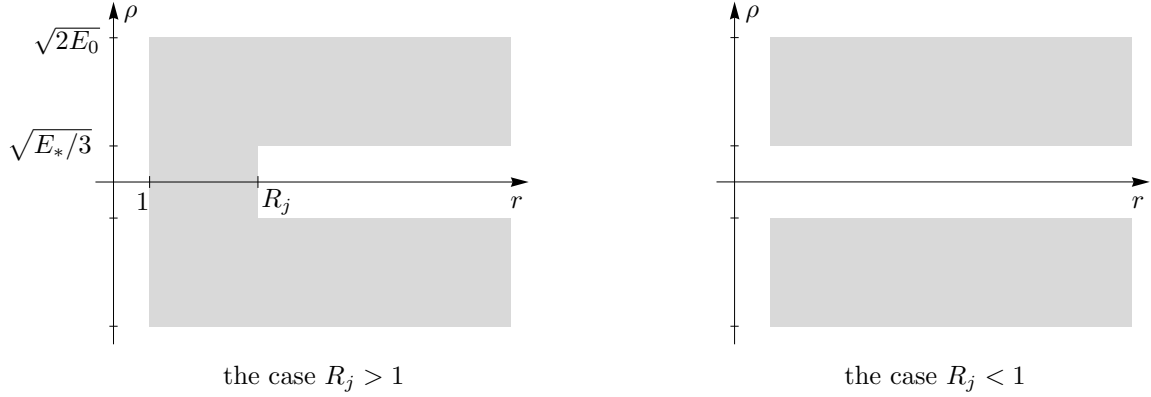

Observe 
that we have $ |p_j - E_j - i \varepsilon| \ge c(1 + \rho^2)$ on $T^*\mathbb R \setminus F_j$, for some $c>0$, which implies the following elliptic estimate: for any $a \in S^m_l, \ a' \in S^{m-2}_l$ satisfying  
$\supp a \cap F_j = \varnothing$ and $|a'(r,\rho)| \ge (1 + |r|)^l(1+|\rho|)^{m-2}$ for $(r,\rho) \in \supp a$, and for any $N \in \mathbb R$, we have 
\begin{equation}\label{e:ellippd}
 \|Au\| \le C\| A'(P_j - E_j -i \varepsilon)u\| + h^N\|Z_N u\|,
\end{equation}
for some $z_N \in S^{m-N}_{l-N}$.
This follows from \eqref{e:compform} by the usual iterative elliptic parametrix construction  as in \cite[Theorem E.32]{dz}.

To handle $F_j$, we define an escape function (based on the usual 
$-r\rho$ but modified to be nonnegative near $F_j$ and more slowly growing) as follows. For $\delta \in (0,1/4)$, take $\tilde q_\delta \in C^\infty(\mathbb R)$  with $\tilde q_\delta(x) = x^\delta$ for $x \ge 2$, $\tilde q_\delta(x) = |x|^{- \delta}$ for $x \le -2$, and $\tilde q_\delta'(x)>0$ for $|x|<2$, 
and put
\begin{equation}\label{e:qdef}
 q (r,\rho):= \tilde q_\delta(-r\rho)\chi_q(r,\rho),
\end{equation}
where  $\chi_q \in S^{-\infty}_0$ is real valued, is $1$ near all of the $F_j$, and vanishes in a neighborhood of
\[
 \{(r,\rho) \mid r \not\in (-1,1 + \max_jR_j) \textrm{ and } \rho =0\}
\]
whose boundary consists of two line segments and four half lines as in Figure \ref{f:chiq}.

\begin{figure}[h]
\labellist
\small
\pinlabel $r$ [l] at 239 25
\pinlabel $\rho$ [l] at 43 70
\pinlabel $-1$ [l] at 10 25
\pinlabel $1+\max_j\!R_j$ [l] at 100 23
\endlabellist
\includegraphics[width=6cm]{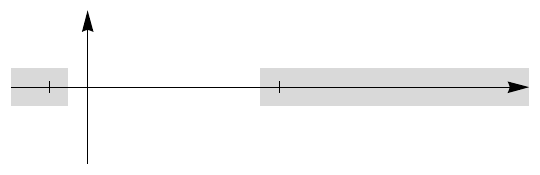}
 \caption{The kind of neighborhood where $\chi_q$ must vanish.}\label{f:chiq}
\end{figure}
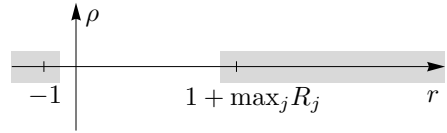

Then $q \in S^{-\infty}_\delta$, and near $F_j$ we have
\begin{equation}\label{e:pbpq}
  \{ \re  p_j, q^2\} = 2(- 2 \rho^2 + r V'_j(r))\tilde q_\delta'(-r\rho)\tilde q_\delta(-r\rho)  \le -c r^{-1-2\delta},
\end{equation}
for some $c>0$ (here we used $V_j \ge E_*/2 \Longrightarrow rV'_j \le -1/C$).

Consequently, there are real valued symbols $b \in S^{-\infty}_{-\frac 12 +\delta}$ and $a_0 \in S^{-\infty}_{-1+2\delta}$ such that
\begin{equation}\label{e:bsym}
  b^2 = \{q^2, \re p_j\}  + a_0,
\end{equation}
 and such that $\supp a_0 \cap F_j = \varnothing$ and $b \ge c r^{-\frac 12 -\delta} >0$ near $F_j$; for example we can take $b := \{ q^2, \re p_j\}^{1/2} \chi_b$ for some $\chi_b \in S^{-\infty}_0$ with $\chi_b = 1$ near $F_j$ and supported in the set where \eqref{e:pbpq} holds. Note that $q$ depends on $\delta$, and $b$ and $a_0$ depend on $\delta$ and $j$, although our notation does not reflect this.

Using \eqref{e:bsym}, \eqref{e:compform}, and \eqref{e:adform}, we can write
\[
 B^*B = \frac i h [Q^*Q, \re P_j] + A_0 + h A_1, 
\]
for some $a_1\in S^{-\infty}_{-2+2\delta} $, giving
\[
 \|Bu\|^2 = \frac i h \langle  [Q^*Q, \re P_j] u, u \rangle + \langle 
 A_0 u, u \rangle + h \langle A_1u, u \rangle,
\]
Combining this with \eqref{e:ellippd} and the similar elliptic estimate
\begin{equation}\label{e:ellipb}
 \| B'u\| \le C \| Bu\| + h ^N\|Z_Nu\|,
\end{equation}
which holds for all $b' \in S^{-\infty}_{-\frac 12 - \delta}$ which is supported in a small enough neighborhood of $F_j$ and for suitable $z_N \in S^{-\infty}_{-\frac 12 - \delta - N}$, 
 we have (since $\delta<1/4$),
\[
 \|(1 + r_+)^{-\frac 12 - \delta}u\|^2 \le C \frac i h \langle  [Q^*Q, \re P_j] u, u \rangle + C \|(P_j - E_j- i \varepsilon)u\|^2.
\]
Next
\[
 i \langle  [Q^*Q, \re P_j] u, u \rangle = - 2 \im \langle Q 
(P_j - E_j  - i \varepsilon)u, Q u \rangle - 2 \re \langle Q (W_e(r)  + \varepsilon)u, Q  u\rangle, 
\]
giving
\[
 \|(1 + r_+)^{-\frac 12 - \delta}u\|^2 \le \frac {C}{h^2} \|(1 + r_+)^{\frac 12 + 3 \delta}(P_j - E_j - i \varepsilon)u\|^2 - \frac C h  
\re \langle Q (W_e(r)   + \varepsilon)u, Q   u\rangle.
\]
But
\[
  - \re \langle Q (W_e(r)  + \varepsilon)u, Q   u\rangle \le |  \re \langle Q ^*[Q ,W_e(r)]u, u \rangle|,
\]
thanks to  $W_e  + \varepsilon \ge 0$, and
 by \eqref{e:compform} and \eqref{e:adform} we have $\re Q^*[Q,W_e(r)] = h^2 a_2$ for some $a_2 \in S^{-\infty}_{-\infty}$, giving
\[
 |  \re \langle Q ^*[Q ,W_e(r)]u, u \rangle| =  h^2 
\langle A_2 u, u \rangle.
\]
This proves \eqref{e:resestjsmallh22} with $s = \frac 12 + 3 \delta$, and taking $\delta>0$ small enough proves it for all $s>1/2$.
\end{proof}

\begin{proof}[Proof of \eqref{e:propsing2}]
Let 
\[
u := (P_j - E_j- i \varepsilon)^{-1} \chi_+(r) \psi(hD_r) v,
\]
 with $\|v\| = 1$, and fix $\delta \in (0,1/4)$. We will use the following argument by induction to prove \eqref{e:propsing2}.

The inductive hypothesis is that for a given $k \in \mathbb R$ there is a neighborhood $U$ of $F_j \setminus (3,\infty)\times(0,\infty)$ such that $\|Au\| \le C h^k$ for any $a\in S^{-\infty}_{k+\frac 12 - \delta}$ which is supported in $U$. 

The inductive step is that there is a (smaller) neighborhood $U'$ of  $F_j \setminus  (3,\infty)\times(0,\infty)$ such that 
\begin{equation}\label{e:uhk}
\|A' u\| \le C h^{k+1/2},
\end{equation}
for any $a'\in S^{-\infty}_{k +1 + \delta}$  which is supported in $U'$. 

Let us see first that \eqref{e:uhk} for arbitrary $k$ implies \eqref{e:propsing2}. Indeed, by the elliptic estimate \eqref{e:ellippd}, the composition formula \eqref{e:compform}, and the resolvent estimate \eqref{e:resestjsmallh22}, we see that 
\begin{equation}\label{e:a''}
\| A'' u\| \le C_N h^N
\end{equation}
for any $N \in \mathbb R$ and $a'' \in S^\infty_{-\infty}$ such that $\supp a'' \subset (0,3) \times \mathbb R$ and $\supp a'' \cap F_j = \varnothing$. Then we can write
\[
 \chi_-(r)u = \chi_-(r)\varphi_F(hD_r)u + \chi_-(r)(1-\varphi_F(hD_r))u
\]
for $\varphi_F \in C_c^\infty(\mathbb R)$ chosen such that \eqref{e:uhk} applies to the first term on the right and \eqref{e:a''} applies to the second. 

We remark in passing that elaborating this argument we can actually show that $u$ is semiclassically trivial everywhere away from the union of two sets (including uniformly as $|r| \to \infty$ and $|\rho| \to \infty$): the first is $\supp \chi_+ \times \supp \psi$, and the second is $F_j \cap  (3,\infty)\times(0,\infty)$ which we can think of as a neighborhood of the forward bicharacteristic flowout of the first. Here we are focusing on a more concrete and narrower version of this conclusion which is sufficient for our purposes.

Next observe that the base case (the inductive hypothesis with $k=-1$ and $U = T^*\mathbb R$) follows from the resolvent estimate \eqref{e:resestjsmallh22}. 

It remains to prove \eqref{e:uhk} under the inductive hypothesis. Roughly speaking, we use an escape function which on  $F_j \setminus  (3,\infty)\times(0,\infty)$ agrees with the one used in the 
proof of \eqref{e:resestjsmallh22} above, but is adapted to vanish  near $\supp \chi_+ 
\times \supp \psi$ and $F_j \setminus U$. (Note that $F_j \setminus U = \varnothing$ when $k=-1$ but that for $k>-1$ we expect $F_j \setminus U \ne \varnothing$ in general).

More specifically, to define the escape function, fix $\chi_k, \psi_k \in C^\infty(\mathbb R)$  nondecreasing, and satisfying $\chi_k = 0$ near $(-\infty,3]$, $\psi_k = 0$ near $(-\infty, 0]$, $\psi_k = 1$ near $[\sqrt{E_*/3},\infty)$, and $\chi_k(r)\psi_k(\rho) = 1$ near $F_j \setminus U$.
Then let
\[
 q_k(r,\rho) := \tilde q_{k+\frac 32 - \delta}(-r\rho)\chi_q(r,\rho)(1 - \chi_k(r)\psi_k(\rho)),
\]
where $\tilde q_{k+\frac 32 - \delta}$ and $\chi_q$ are as in \eqref{e:qdef}, so that $ q_k \in S^{-\infty}_{k+\frac 32 - \delta}$. Calculating as in \eqref{e:pbpq}, we see that near $F_j$ we have
\[
 \{\re p_{j}, q_k^2\} \le 0 ,
\]
and near $F_j \setminus  (3,\infty)\times(0,\infty)$ we have $\chi_k (r)\psi_k(\rho) = 0$ and hence $ \{\re p_{j}, q_k^2\} \le -c r^{2k+2 - 2 \delta} < 0$  (this is slightly better than \eqref{e:pbpq} because outside of a compact set we have $\rho<0$ on $F_j \setminus  (3,\infty)\times(0,\infty)$ and in particular we are staying away from the outgoing part of the energy surface).

Consequently, as before, we can write
\[
 b_k^2 =  \{ q_k^2, \re p_j\} + a_{0,k},
\]
where $ b_k \in S^{-\infty}_{k+1-\delta},\ a_{0,k} \in S^{-\infty}_{2k+2 - 2 \delta}$, $\supp b_k\ \subset \supp q_k$, $\supp a_{0,k} \cap (F_j \cup \supp \chi_+ 
\times \supp \psi) = \varnothing$, and $b_k \ge cr^{k+1-\delta} >0$ near  $F_j \setminus  (3,\infty)\times(0,\infty)$. Hence
\[
 B_k^*B_k = \frac i h [Q_k^*Q_k, \re P_j] + A_{0,k} + h A_{1,k},
\]
for some $a_{1,k} \in  S^{-\infty}_{2k+1 - 2 \delta}$. We refine this by using \eqref{e:compform} and \eqref{e:adform} to expand $a_{1,k}$ in powers of $h$ up to $h^N$ in terms of $b_k$, $q_k$, $p_j$, $a_{0,k}$, and their derivatives, which gives
\[
 B_k^*B_k = \frac i h [Q_k^*Q_k, \re P_j] + A_{0,k} + h A'_{1,k} + h^N Z_N,
\]
where $a'_{1,k}\in  S^{-\infty}_{2k+1 - 2 \delta}$ has $\supp a'_{1,k} \subset \supp q_k$ and $z_N \in S^{-\infty}_{2k+2 - 2 \delta - N}$. Consequently
\[
 \|B_ku\|^2 = \frac i h \langle  [Q_k^*Q_k, \re P_j] u, u \rangle + \langle 
 A_{0,k}u, u \rangle + h \langle A_{1,k}' u, u \rangle + h^N\langle Z_N u, u \rangle.
\]
 By the elliptic estimate \eqref{e:ellipb} with $b_k$ in place of $b$  we see that to deduce \eqref{e:uhk} it is enough to show
\begin{equation}\label{e:bkbd}
  \|B_ku\|^2 \le C h^{2k+1}.
\end{equation}
Now $\langle A_{0,k} u, u \rangle = O(h^\infty)$ by \eqref{e:ellippd}. Also, since $q_k$ vanishes near $F_j \setminus U$, it follows that $a_{1,k}'$ vanishes near $F_j \setminus U$, so by \eqref{e:ellippd}, \eqref{e:compform}, and the inductive hypothesis, we have
\[
 |\langle A_{1,k}' u, u \rangle| \le C h^{2k}.
\]
Hence to show \eqref{e:bkbd} it suffices to show that
\begin{equation}\label{e:qqp}
 i \langle  [Q_k^*Q_k, \re P_j] u, u \rangle \le  Ch^{2k+2} .
\end{equation}

As before we write, for any $N \in \mathbb R$,
\[\begin{split}
 i \langle  [Q_k^*Q_k, \re P_j] u, u \rangle &= - 2 \im \langle Q_k 
(P_j - E_j  - i \varepsilon)u, Q_ku \rangle - 2 \re \langle Q_k( W_e(r)   
+ \varepsilon)u, Q_k  u\rangle \\
& \le 2 |\re \langle Q_k^*[Q_k,W_e(r)]u,u\rangle| + O(h^\infty),
\end{split}\]
where we used $\supp q_k \cap \supp \chi_+ \times \supp \psi = \varnothing$. Now \eqref{e:qqp} follows from the inductive hypothesis together with the fact that (arguing as in the construction of $a_{1,k}'$ above) $\re Q_k^*[Q_k,W_e(r)] = h^2 A_{2.k} + h^N Z_N$, with $a_{2,k}, \ z_N \in S^{-\infty}_{-\infty}$, and $\supp a_{2,k} \cap F_j \subset U$.
\end{proof}

\subsection{Proof of Theorems \ref{t:cont} and \ref{t:away}}\label{s:gluing}

In this section all operator norms are $L^2(X) \to L^2(X)$. We implement the outline discussed in \S\ref{s:outline}. We assume without loss of generality that $\varepsilon \in (0,1]$, as the statements with $\varepsilon >1$ follow from self-adjointness and the statements with $\varepsilon<0$ then follow by taking the adjoint.

 We first explain the key dynamical property of the bicharacteristic flow in $X_e$ which allows us to remove the remainders in the parametrix construction.

Let us denote points in $T^*X_e$ by $(r,y,\rho,\eta)$, where $y \in Y$, $\rho$ is dual to $r$, and $\eta$ is dual to $y$. The \textit{energy surface} for $P$ in $T^*X_e$ at energy $E_0$ is the subset of $T^*X_e$ defined by
\[
 p(r,y,\rho,\eta):= \rho^2 + |\eta|^2 f(r)^{-4/(d-1)} + V_L(r) = E_0,
\]
and \textit{bicharacteristics}  in $T^*X_e$ of this energy surface are solutions  $\gamma(t):=((r(t),y(t),\rho(t),\eta(t))$ to the Hamiltonian equation of motion $\dot \gamma(t) := \frac d {dt} \gamma(t) = \{p,\gamma(t)\}$. The \textit{backward bicharacteristic flowout}  in $T^*X_e$ of a point $\gamma_0 \in T^*X_e$ is the set of points $\gamma' \in T^*X_e$ such that if $\gamma(t)$ is the bicharacteristic  in $T^*X_e$ with $\gamma(0) = \gamma_0$, then $\gamma(t) = \gamma'$ for some $t \le 0$; note that some bicharacteristics  enter $T^*(X\setminus X_e)$ in finite time, and our definition only counts them while they stay in $T^*X_e$.

If $\gamma(t):=((r(t),y(t),\rho(t),\eta(t))$  is a bicharacteristic, then
\begin{equation}\label{e:rdot}
 \dot r(t) = 2 \rho(t), \qquad \dot \rho(t) =  \frac 4{d-1} |\eta|^2 f'(r(t))f(r(t))^{-(d+3)/(d-1)} - V_L'(r(t)) \ge 0,
\end{equation}
and hence $\ddot r = 2 \dot \rho \ge 0$. Consequently no bicharacteristic can visit the sets $T^*((0,4))$, $T^*((4,5))$, and $T^*((2,3))$ in that order (here and below $T^*((a,b))$ denotes the subset of $T^*X_e$ on which $a<r<b$), and this fact is exploited to prove the crucial remainder estimate in    \eqref{e:acak} below.

 Fix $\chi_e, \chi_K \in C^\infty(\mathbb R)$ such that $\chi_e + \chi_K = 1$, 
$\supp \chi_e \subset (3,\infty)$, and $\supp \chi_K \subset (-\infty,4)$. Define a parametrix for $P - E - i \varepsilon$ by
\[
G:= \chi_K(r-1)R_K\chi_K(r) + \chi_e(r+1)R_e \chi_e(r).
\]
Here
\[
 R_K = R_K(E_0+ i \varepsilon):= (-h^2\Delta - i W_K(r) - E_0 - i \varepsilon)^{-1},
\]
and
\[
R_e = R_e(E_0 + i \varepsilon):= f(r) \sum_{j=0}^\infty\left ((P_j - i \varepsilon)^{-1} \phi_j \otimes \phi_j \right)
f(r)^{-1},
\]
and
\begin{equation}\label{e:rcrkbd}
 \|R_K\| \le C a(h)h^{-1}, \qquad \|(1+r)^{-s_1}\chi_e(r+1)R_e \chi_e(r)(1+r)^{-s_2}\| \le C h^{-2}.
\end{equation}
Indeed, $R_K$ is well defined and obeys \eqref{e:rcrkbd} thanks to 
 \eqref{e:w}; this follows from the resolvent identity for 
$\varepsilon >0$ small enough and then from the bound $\im (-h^2 \Delta - i W_K(r) 
- E_0 - i \varepsilon) \le - \varepsilon$ for all $\varepsilon>0$. Meanwhile   $ \chi_e(r+1)R_e \chi_e(r)$ acts on $L^2(X)$ thanks to \eqref{e:pccagree} and the support property of $\chi_e$, even though $R_e$ acts on a funny space due to the way we defined the operators $P_j$ differently depending on $j$; moreover
$R_e$ obeys \eqref{e:rcrkbd} by \eqref{e:resestjbig} and 
\eqref{e:resestjsmallh22}.
  
Define operators $A_K$ and $A_e$ by
\[
(P - E_0 - i \varepsilon)G = I + [h^2 D_r^2, \chi_K(r-1)]R_K\chi_K(r) +  [h^2 
D_r^2, \chi_e(r+1)]R_e\chi_e(r) =: I + A_K + A_e.
\]

Our next step is to remove the remainders $A_K$ and $A_e$. The idea of \cite{dv} is to do this using a \textit{semiclassically outgiong} property of the resolvents $R_K$ and $R_e$. 

To explain this property, we use the following notation: if $U \subset T^*X_e$, then $\Gamma_+U$ is the set of points in $T^*X_e$ whose backward bicharacteristic flowout intersects $U$. Now in the case of $R_K$, the needed semiclassically outgoing property says (in the notation of \eqref{e:scdef} and \eqref{e:quantdef}) that if $\tilde \chi \in C_c^\infty((0,\infty))$ and $a \in S^{0}_l$, then
\begin{equation}\label{e:rkog}
 \|\tilde \chi(r)\Op_h(a) A_K\| = O(h^\infty),
\end{equation}
provided $|\partial_r^{n_1}\partial_\rho^{n_2} a(r,\rho)| = O(h^\infty)$ for every $n_1, \ n_2 \in \mathbb N_0$ and for every $(r,\rho) \in T^*((0,4)) \cup \Gamma_+T^*((0,4))$. 
This property follows from \cite[Lemma 5.1]{dv}. 

On the other hand, the resolvent $R_e$ is only semiclassically outgoing for $j$ such that $E_j \ge c >0$ (the relevant statement for us is \eqref{e:propsing2}); as $E_j \to0$ this property fails, but then the gluing region (the part of $X$ such that $r \in (2,5)$) becomes classically forbidden, and so we will be able to estimate and remove remainders using the Agmon estimates of \S\ref{s:jbig}.

More specifically, we observe that
\begin{equation}\label{e:akacbd}
 \|A_K\| \le C(1+ a(h)), \qquad  \|A_e(1+r)^{-s_2}\| \le C.
\end{equation}
Indeed, $A_K$ obeys the bound thanks to the corresponding bound on $R_K$ in \eqref{e:rcrkbd}; note that $\|R_K\|_{L^2 \to H^2_h(X)} \le C \|R_K\|$ since $V$, $W$, and $\varepsilon$ are bounded, and $E_0$ is fixed. Meanwhile $A_e$ obeys the bound by \eqref{e:agmonoverlap} and \eqref{e:resestjsmallh22}. 

We refine the parametrix with some correction terms, observing that $A_K^2 = 
A_e^2 = 0$:
\[
(P - E_0 - i \varepsilon)G(I - A_K - A_e + A_KA_e) = I  - A_eA_K + A_eA_KA_e.
\]
We will show that
\begin{equation}\label{e:acak}
\|A_eA_K\| = O(h^\infty).
\end{equation}

Assuming \eqref{e:acak} for the moment, we may  write (using $R_e \chi_e(r) A_e = R_K \chi_K(r) A_K = 0$)
\begin{equation}\label{e:pinv}\begin{split}
(P - E_0 -& i \varepsilon)^{-1} = G(I - A_K - A_e + A_KA_e)(I  - A_eA_K + A_eA_KA_e)^{-1} \\
  =&\chi_e(r+1)R_e \chi_e(r) + \chi_K(r-1)R_K\chi_K(r)  - \chi_e(r+1)R_e A_K \\ &- \chi_K(r-1)R_KA_e  + \chi_e(r+1)R_e A_KA_e + O(h^\infty).
\end{split}\end{equation}
Note that by by \eqref{e:agmonoverlap},  \eqref{e:resestjsmallh22}, and the bound on $\|R_K\|$ in \eqref{e:rcrkbd}, we have
\begin{equation}\label{e:rcak}
 \|(1+r)^{-s_1}\chi_e(r+1)R_eA_K\| \le C a(h)h^{-1}.
\end{equation}
Now multiplying \eqref{e:pinv} on the left by $(1+r)^{-s_1}$ and on the right by $(1+r)^{-s_2}$ and estimating the norm on the right term by term, we see that by \eqref{e:rcrkbd} the first term on the right has norm bounded by $Ch^{-2}$, while by \eqref{e:rcrkbd},  \eqref{e:akacbd}, and \eqref{e:rcak}, the next four terms have norm bounded by $Ca(h)h^{-1}$. This implies \eqref{e:2n-1}.

We similarly deduce \eqref{e:taway} from \eqref{e:pinv}, but rather than using the bound on $R_e$ in \eqref{e:rcrkbd}, we use
\begin{equation}\label{e:rechij}
\|(1+r)^{-s} \chi_e(r+1) \chi_{\mathcal J} R_e \chi_e(r)(1+r)^{-s}\|\le Ch^{-1}.
\end{equation}
To prove \eqref{e:rechij}, we use \eqref{e:resestjbigchi} when $E_j \in [-c_{\mathcal J} h, c_{\mathcal J}]$, we use \eqref{e:resestjsmallh22} when $E_j \ge c_{\mathcal J}$, and we use the fact that $P_j$ is almost nonnegative (more precisely, $P_j \ge - Ch^2$ by \eqref{e:pccagree} and \eqref{e:pjbigdef}) when $E_j \le - c_{\mathcal J} h$.

To complete the proofs of Theorems \ref{t:cont} and \ref{t:away}, it remains to show \eqref{e:acak}. We have
\[
A_eA_K =  [\chi_e(r+1),h^2 D_r^2]R_e [\chi_K(r-1),h^2 D_r^2]R_K\chi_K(r). 
\]
Fix $\tilde \chi \in C_c^\infty((3,6))$ which is $1$ on $[4,5]$, so that
\[
A_eA_K =  [\chi_e(r+1),h^2 D_r^2] R_e \tilde \chi(r) [\chi_K(r-1),h^2 
D_r^2]R_K\chi_K(r). 
\]
For any $\psi \in C_c^\infty((0,\infty))$ we have
\[
\| [\chi_e(r+1),h^2 D_r^2]R_e  \tilde \chi (r) \psi(hD_r) [\chi_K(r-1),h^2 
D_r^2] \| = O(h^\infty),
\]
by \eqref{e:agmondisjoint} and \eqref{e:propsing2}, so it remains to show that there is $\psi \in C_c^\infty((0,\infty))$ such that
\[
 \| \tilde \chi(r) (I - \psi(hD_r)) [\chi_K(r-1),h^2 D_r^2]R_K\chi_K(r)\| = 
O(h^\infty).
\]
We will deduce this from \eqref{e:rkog}. Indeed, it is enough to check that there is $\rho_0>0$ such that if $\gamma(t)$ is a bicharacteristic at energy $E_0$
with $\gamma(0) \in T^*\supp\chi_K(r)$ and with $\gamma(T) \in T^*\supp \chi_K'(r-1)$ for some $T>0$, then $\rho(T)\ge \rho_0$ (we already know that $\rho(T)^2 \le E_0$, so we may then take $\psi$ to be $1$ near $[\rho_0,\sqrt{E_0}]$).

 Thanks to \eqref{e:rdot} we know that $\rho(t)$ is nondecreasing, so we may assume that $\max \supp \chi_K(r) <r(t) <\min\supp \chi'_K(r-1)$ when $t \in (0,T)$, which implies in particular $\rho(0) \ge 0$. Then, for $t \in (0,T)$, we have $f(r(t)) \le C f'(r(t))$ and $V_L(r(t)) \le  - C V'_L(r(t))$, so that
\[
 \dot \rho(t) \ge   (|\eta|^2 f(r)^{-4/(d-1)} + V_L(r))/C_0 = (E_0 -\rho(t)^2)/C_0.
\]
If $\rho(0) = \sqrt{E_0}$, then $\rho(T) = \sqrt{E_0}$  and we are done; otherwise we can integrate and use $\rho(0) \ge 0$ to obtain
\[
\frac {C_0}{\sqrt {E_0}} \tanh^{-1}\left(\frac{\rho(T)}{\sqrt{E_0}}\right) \ge  T =  \frac{r(T) - r(0)}{2\bar \rho} \ge  \frac{r(T) - r(0)}{2 \rho(T)},
\]
where we used $\bar \rho: = T^{-1}\int_0^T \rho(t)dt \le \rho(T)$. This implies $\rho(T) \ge \rho_0$, for some $\rho_0>0$ depending on $C_0, \ E_0,$ and $\chi_K$.

\section{Continuation of the resolvent}\label{s:continuation}

In this section we keep all of the assumptions of \S\ref{s:assump}, and add the assumption that 
\[
 r \ge 6 \Longrightarrow V_L(r) = f(r) - 1 = 0.
\]

In \S\ref{s:mero} we  briefly review how meromorphic continuation works in this setting, following \cite{gui} and \cite[\S6.7]{mel}, and introduce the relevant notation. In \S\ref{s:modelcyl} we prove some useful estimates for a model problem on the cylindrical end. In \S\ref{s:resfree} we use an identity of Vodev from \cite{v} to deduce the existence of a resonance free region.

Roughly speaking, writing $R(z)$ for the resolvent $(P-z)^{-1}$ and for its meromorphic continuation, we deduce from \eqref{e:2n-1} that
\[
 \|\chi R(E_0\pm i0)\chi\| \lesssim 1/\mu(h),
\]
where $\chi \in C_c^\infty(X)$ and $0<\mu(h)\le h^2$. Then we use Vodev's identity to show that this implies
\[
 \|\chi R(z)\chi\| \lesssim 1/\mu(h),
\]
as long as the distance from $z$ to $E_0 \pm i0$ is small compared to $\mu(h)$. However some care is needed due to the complicated nature of the Riemann surface to which $R(z)$ continues (see \S\ref{s:mero}), and due to the fact that our model resolvent obeys somewhat weaker bounds than the one used in \cite{v} (see \S\ref{s:modelcyl}). The precise statement and proof are in \S\ref{s:resfree}.

Although we keep all of the assumptions of \S\ref{s:assump} in this section, strictly speaking they are not all needed once we have \eqref{e:2n-1}. Instead, as long as we had \eqref{e:2n-1}, we could allow $X$ to be a more general manifold with cylindrical ends, or allow $P$ to be a black-box perturbation of the Laplacian e.g. in the sense of \cite[\S2]{cd}. The proof could also be adapted to include the case of waveguides. We omit these generalizations here, to simplify the presentation and because all of our interesting examples satisfy the assumptions of \S\ref{s:assump}.

\subsection{Meromorphic continuation of the resolvent}\label{s:mero} In \S\ref{s:mero} we think of $h>0$ as being fixed, until Lemma \ref{l:projdiff}, in which we prove an estimate which is uniform as $h \to 0$.

The spectrum of $P$ is given by $[0,\infty)$ together with a finite (possibly empty) set of negative eigenvalues. For $z$ not in the spectrum we define the resolvent
\[
 R(z):=(P-z)^{-1} \colon L^2(X) \to L^2(X).
\]
To define the Riemann surface onto which $R(z)$ meromorphically continues, for each $j \in \mathbb N_0$, and $z \in \mathbb C \setminus [h^2\sigma_j^2,\infty)$, we introduce the notation
\[
 \rho_j(z) := \sqrt{z-h^2\sigma_j^2},
\]
with the branch of the square root chosen such that $\im \rho_j(z)>0$ for this range of $z$ (recall that $0=\sigma_0 \le \sigma_1 \le \cdots$ are the square roots of the eigenvalues of the nonnegative Laplacian on $(Y,g_Y)$ included according to multiplicity). 

For each  $j \in \mathbb N_0$, there is a minimal Riemann surface $\hat Z_{h,j}$ onto which $\rho_j$ continues analytically from $\mathbb C \setminus [h^2\sigma_j^2,\infty)$; this 
is a double cover of $\mathbb C$ ramified at the singular point $z = h^2\sigma_j^2$. By elaborating the construction of $\hat Z_{h,j}$, we see that  there is a minimal Riemann surface $\hat Z_h$ onto which all the $\rho_j$ extend simultaneously from $\mathbb C \setminus [0,\infty)$. This is a countable cover of $\mathbb C$, ramified at $z=h^2\sigma_j^2$ for each $j$, and for each  $z \in \hat Z_h$ we have $\im \rho_j(z)>0$ for all but finitely many $j$. For more details, see \cite{gui} and \cite[\S 6.7]{mel}.

We use $\pr$ to denote the projection $\hat Z_h \to \mathbb C$, we use the term \textit{physical region} to refer to the sheet over $\mathbb C \setminus [0,\infty)$ on which $\im \rho_j>0$ for all $j$, and for notational convenience we identify the physical region with $\mathbb C \setminus [0,\infty)$. Then $R(z)$ continues meromorphically from the resolvent set in $\mathbb C \setminus [0,\infty)$ to all of $\hat Z_h$, as an operator from compactly supported $L^2$ functions to locally $L^2$ functions, and we have $(P-\pr(z))R(z) = I$. We refer to the poles of $R(z)$ as \textit{resonances}.

For $E \ge 0$, we denote by $E\pm i0$ the points in $\hat Z_h$ on the boundary of the physical region which 
are obtained as limits $\lim_{\pm \delta \downarrow 0} E+i\delta$.  Note that $\rho_j(E\pm i0) \in i\Real_+$ if $E<h^2 \sigma_j^2$, and $\pm \rho_j(E\pm i0)>0$ if $h^2\sigma^2_j<E$. Below we will only be concerned with points on $\hat Z_h$ which are quite close to the boundary of the physical region.
To measure  how far apart two points on $\zhath$ are we use the following
\begin{lem} The function $d_h:\zhath\times \zhath\rightarrow [0,\infty]$ given by
\begin{equation}\label{e:dhdef}
d_h(z,z'):=\sup_j |\rho_j(z)-\rho_j(z')|
\end{equation}
takes only finite values and is a metric on $\zhath$.
\end{lem}

\begin{proof} To see that $|\rho_j(z)-\rho_j(z')|$ is bounded in $j$, note that 
\begin{equation}\label{eq:projdif}
\pr(z)-\pr(z')=\rho^2_j(z)-\rho^2_j(z')=(\rho_j(z)-\rho_j(z'))(\rho_j(z)+\rho_j(z')).
\end{equation}
Using that $\rho_j^2(z)= \pr(z)-h^2\sigma_j^2$, we find
 $\re \rho_j^2(z) \rightarrow -\infty$ as $j\rightarrow \infty$.  Since $\im \rho_j(z)>0$ if $j$ is sufficiently large, $\im \rho_j(z)\rightarrow \infty$ as $j\rightarrow \infty$ and we find, since the same is true for 
$z'$, that  for $j$ large enough $|\rho_j(z)-\rho_j(z')|< |\rho_j(z)+ \rho_j(z')|$.  Since by 
(\ref{eq:projdif}), we have 
\[
\min\{|\rho_j(z)-\rho_j(z')|,|\rho_j(z)+\rho_j(z')|\} \le | \pr(z)-\pr(z')|^{1/2},
\]
 we have for $j$ sufficiently large, $|\rho_j(z)-\rho_j(z')|\leq |\pr(z)-\pr(z')|^{1/2}$.

That $d_h$ is a metric is fairly straightforward;  for completeness we check  the triangle inequality.
Let $z,\; z', w\in \zhath$.  Then
$$|\rho_j(z)-\rho_j(z') | \leq |\rho_j(z)-\rho_j(w)|+|\rho_j(w)-\rho_j(z')|.$$
But then
\begin{align*} 
d_h(z,z')& =\sup_j |\rho_j(z)-\rho_j(z') | \leq \sup_j (|\rho_j(z)-\rho_j(w)|+|\rho_j(w)-\rho_j(z')| )\\ & 
\leq \sup_j  |\rho_j(z)-\rho_j(w)|+ \sup_j|\rho_j(w)-\rho_j(z')| = d_h(z,w)+d_h(w,z').
\end{align*}
\end{proof}
Later we will want to use $d_h(z,z')$ in a resolvent identity, and now we show that $d_h(z,z')$ controls $|\pr(z)-\pr(z')|$, at
least when $z'$  is on the boundary of the physical region:
\begin{lem} \label{l:projdiff} Let $E>0$, and let $E\pm i0$ denote one of the points on the boundary of the physical space in 
$\zhath$  as described above.  Then for any $\delta>0$, if $h>0$ is sufficiently small,
$$|\pr(z)-E | \leq d_h(z,E\pm i0)[ d_h(z,E\pm i0)+O(h^{1/2-\delta})]$$
for $z\in \zhath$.
\end{lem}
\begin{proof}
We have, for any $j\in \Natural$,
\begin{align}\label{eq:projectionbd}
|\pr(z)- E |& = |\rho_j^2(z)-\rho_j^2(E\pm i0)| \nonumber \\ &
= |\rho_j(z)-\rho_j(E\pm i0)) | |\rho_j(z)- \rho_j(E\pm i0)+ 2\rho_j(E\pm i0)| \nonumber \\ &
\leq |\rho_j(z)-\rho_j(E\pm i0)| \left( | \rho_j(z)- \rho_j(E\pm i0)|+ 2|\rho_j(E\pm i0)|\right).
\end{align}
By the Weyl law, for any $\delta'>0$ there is an $h_0=h_0(\delta')>0$ so that 
if $0<h<h_0$, the interval $[Eh^{-2}- h^{-1-\delta'}, Eh^{-2}+ h^{-1-\delta'}] $ contains an element of 
the spectrum of $-\Delta_Y$; call this $\sigma_{j_0}^2$.  
We note that $j_0$ depends on $E$ and on $h$, but our notation does not reflect that dependence.
Then 
$$|\rho_{j_0}(E\pm i0)|^2=|E- h^2\sigma_{j_0}^2|\leq h^{1-\delta'}.$$
Using this in 
(\ref{eq:projectionbd}) with $j=j_0$ proves the lemma, since 
$|\rho_{j_0}(z)-\rho_{j_0}(E\pm i0)|\leq d_h(z,E \pm i0)$.
\end{proof}


\subsection{Resolvent estimates for the model problem on the cylindrical end.}\label{s:modelcyl}
Let $X_0 =[0,\infty)\times Y $,  let $\Delta_0 \le 0$ be the Laplacian on $(X_0, dr^2 + g_Y)$, and for $h>0$ and $z \in \mathbb C \setminus [0,\infty)$, let
\[
 R_0(z) := (-h^2\Delta_0 - z)^{-1}
\]
denote the semiclassical Dirichlet resolvent. 

For $\im \xi>0$, let $\tR(\xi)$ be the resolvent for the Dirichlet Laplacian on the half-line with spectral parameter $\xi^2$ and  Schwartz kernel given by 
\begin{equation}\label{e:resdhalf}
\tR(\xi,r,r') = \frac{i}{2h \xi} (e^{i\xi|r-r'|/h}- e^{i\xi(r+r')/h}).
\end{equation}
Then, for $z$ in the physical region of $\hat Z_h$ (see \S\ref{s:mero}), we have
\begin{equation}\label{eq:R0}
R_0(z)= \sum_{j=0}^\infty  \tR(\rho_j(z))  \phi_j \otimes\phi_j,
\end{equation}
where  $\{ \phi_j\}_{j=0}^\infty$ is a complete set of real-valued orthonormal eigenfunctions of the Laplacian on $Y$ and $-\Delta_Y \phi_j=\sigma_j^2 \phi_j$.

Moreover, $R_0(z)$ continues holomorphically to $\hat Z_h$  as an operator from compactly supported $L^2$ functions to locally $L^2$ functions. 
In this section we prove some estimates for $R_0(z)$ which will be needed when we use a resolvent identity to find a neighborhood of the boundary of the physical region in which $R(z)$ has no poles.

\begin{prop}\label{p:trdiffbds}
Let $\chi\in C^\infty_c([0,\infty))$ and fix $N>0$.  If $\im \xi, \ \im \xi' >-Nh$, then
\begin{equation}\label{e:rd1}
\| \chi \tR(\xi)\chi - \chi \tR(\xi')\chi\| \leq  C h^{-3}|\xi -\xi'|.
\end{equation}
If $\im \xi, \ \im \xi' >-Nh$ and $\alpha_1+\alpha_2=1,2$, then 
\begin{equation}\label{e:rd2}
\| \chi h^{\alpha_1}D_r^{\alpha_1} \tR(\xi)h^{\alpha_2}D_r^{\alpha_2}\chi - \chi h^{\alpha_1}D_r^{\alpha_1} \tR(\xi')h^{\alpha_2}D_r^{\alpha_2}\chi\|\leq C h^{-2}|\xi -\xi'| (|\xi|+|\xi'|+1)^{\alpha_1+\alpha_2-1}.
\end{equation}
Fix $\delta>0$ and suppose $\delta<\arg \xi,\arg \xi'<\pi-\delta$ and $|\xi|, |\xi'|\geq 1$.  Then if $\alpha_1+\alpha_2\leq 2$,
\begin{equation}\label{e:rd3}
\| h^{\alpha_1}D_r^{\alpha_1} \tR(\xi)h^{\alpha_2}D_r^{\alpha_2}\chi - \chi h^{\alpha_1}D_r^{\alpha_1} \tR(\xi')h^{\alpha_2}D_r^{\alpha_2}\|\leq C |\xi -\xi'|.
\end{equation}
All the norms above are $L^2(\mathbb R_+) \to L^2(\mathbb R_+)$, and the constants depend on $\chi$, $N$, and $\delta$.
\end{prop} 

\begin{proof}
We begin with \eqref{e:rd1}.  Note that $\chi\frac{d}{d\xi}\tR(\xi)\chi$ has Schwartz kernel
\[
\frac{i\chi(r)}{2 h^3 (\xi/h)^2} \left( (-1+i|r-r'|\frac{\xi}{h}) e^{i\xi|r-r'|/h}-(-1+i(r+r')\frac{\xi}{h})e^{i\xi(r+r')/h}\right)\chi(r').
\]
With $\im \xi >-Nh$, this can be pointwise bounded by $C/h^3$, even when $\xi\rightarrow 0$, and 
hence  since $\chi$ is compactly supported we have $\| \chi\frac{d}{d\tau}\tR(\tau)\chi\| \leq \frac{C}{h^3}$.
Integrating from $\xi$ to $\xi'$ gives \eqref{e:rd1}.  We note for future reference that if $|\xi|\geq h$, then we can improve the estimate to 
\begin{equation}\begin{split}\label{eq:improvebigxi}
\left|\frac{i\chi(r)}{2 h^3 (\xi/h)^2} \left( \left(-1+i|r-r'|\frac{\xi}{h}\right) e^{i\xi|r-r'|/h}-\left(-1+i(r+r')\frac{\xi}{h}\right)e^{i\xi(r+r')/h}\right)\chi(r')\right|&  \\
\leq C/(h^2 |\xi|), \text{ when } |\xi|&\geq h.
\end{split}\end{equation}

Next consider the operator $h\frac{\partial }{\partial r} \tR(\xi)$.  It has Schwartz kernel
$$  \frac{-1}{2h} \left( \sgn(r-r')e^{i\xi|r-r'|/h}- e^{i\xi(r+r')/h}\right).$$
Differentiating this with respect to $\xi$ and proceeding as above gives
$\left\| \chi\frac{d}{d\tau}h\frac{\partial}{\partial r} \tR(\tau)\chi\right\| \leq \frac{C}{h^2}.$
Integrating in $\tau$ from $\xi$ to $\xi'$ gives \eqref{e:rd2} for $\alpha_1=1$, $\alpha_2=0$.  
To prove \eqref{e:rd2} for $\alpha_1= 2$, $\alpha_2=0,$ we can argue as before using the Schwartz kernel.  Alternately, we can note that 
$h^2\frac{\partial^2}{\partial r^2}\tR(\xi)= I + \xi^2 \tR(\xi)$ and proceed as in the proof of the first inequality, using the improvement (\ref{eq:improvebigxi}). Similar techniques give \eqref{e:rd2} when $\alpha_2 \ne 0$, if we consider the Schwartz kernel of $R_D(\xi)\frac \partial{\partial_r}$.

  When $\xi, \; \xi'$ satisfy $\delta<\arg \xi,\arg \xi'<\pi-\delta$ they are both in the physical region and we can use the resolvent equation
$\tR(\xi)-\tR(\xi')= (\xi^2- \xi'^2 ) \tR(\xi)\tR(\xi')$.    If $|\xi|\geq 1$, using the bound on $\arg \xi$ we have $\| h^{\alpha_1}D_r^{\alpha_1} \tR(\xi)h^{\alpha_2}D_r^{\alpha_2}\|\leq C|\xi|^{\alpha_1+\alpha_2-2}$, where the constant depends on $\delta$.  The same inequality holds if $\xi$ is replaced by $\xi'$ everywhere.  Using this in the resolvent
equation proves~\eqref{e:rd3}.
\end{proof}

\begin{prop} \label{p:resest} Let $E>0$ and consider one of the points
$E\pm i0\in \zhath$ which lies on the boundary of the physical region.
 Fix $N>0$ and  $\chi \in C_c^\infty(X_0)$.  Then
\begin{equation}\label{eq:resest1} \| \chi R_0(z)\chi - \chi R_0(E \pm i0)\chi \| \leq C h^{-3}d_h(z,E \pm i0),
\end{equation}
for all $z \in \zhath$ such that $d_h(z,E\pm i 0)<Nh$.
If $\alpha_1+\alpha_2=1,2$, then instead
\begin{equation}\label{eq:resest2}\| \chi h^{\alpha_1}D_r^{\alpha_1} R_0(z)h^{\alpha_2}D_r^{\alpha_2}\chi - \chi h^{\alpha_1}D_r^{\alpha_1} R_0(E \pm i 0) h^{\alpha_2}D_r^{\alpha_2}\chi\|\leq C h^{-2}
d_h(z,E \pm i0),\end{equation}
for all $z \in \zhath$ such that $d_h(z,E\pm i 0)<Nh$.
\end{prop}
\begin{proof} 
We begin by noting that for any $j\in \Natural $, $\im \rho_j(E\pm i0) \geq 0$, and for $h^2\sigma_j^2>E$ we have $\rho_j(E\pm i0)\in i\Real_+$.
Hence if $d_h(z,E \pm i0)<Nh$, then $\im \rho_j(z)\geq -Nh$ and $\im \rho_j(z)\rightarrow \infty$
as $j\rightarrow \infty$.

Without loss of generality, we may assume $\chi$ is a function of $r$ only, so that we may consider $\chi$ as a function defined on $[0,\infty)$.
Using the expression (\ref{eq:R0}), we find that 
\[
\| \chi R_0(z)\chi- \chi R_0(E \pm i0)\chi\|_{L^2(X_0)\rightarrow L^2(X_0)}
 = \sup_j \| \chi \tR(\rho_j(z))\chi- \chi \tR(\rho_j(E \pm i0))\chi\|_{L^2(\Real_+)\rightarrow L^2(\Real_+)}.\]
Now \eqref{eq:resest1} follows directly from \eqref{e:rd1} and the definition \eqref{e:dhdef} of $d_h(z,E \pm i0)$.

To prove \eqref{eq:resest2},  we note that for $j$ sufficiently large we have $h^2 \sigma^2_j> E+5$, and $\pi/4<\arg \rho_j(z), \arg \rho_j(E \pm i0)< 3\pi/4$.   Using \eqref{e:rd2} when $h^2 \sigma^2_j\leq  E+5$ and \eqref{e:rd3}  when $ h^2 \sigma^2_j> E+5$,
along with the definition of $d_h(z,E \pm i0)$ proves \eqref{eq:resest2}.
\end{proof}

\subsection{The resonance free region} \label{s:resfree}
Throughout \S\ref{s:resfree},  we keep all of the assumptions of \S\ref{s:assump}, as well as the assumption that 
\[
 r \ge 6 \Longrightarrow V_L(r) = f(r) - 1 = 0.
\]

To show the existence of a resonance free region, we use an identity due to Vodev
\cite[(5.4)]{v}.  In \cite{v} the identity is stated only for operators which are potential perturbations of the Laplacian on $\Real^d$.  However, it in fact holds in far greater
generality for operators which are, in an appropriate sense, compactly supported perturbations of each other.   Here we state a version adapted to our circumstance.  

\begin{lem} \label{l:vodev}
(\cite[(5.4)]{v})  
Let $\chi _1\in C_c^{\infty}(X;[0,1])$ be such that $r \ge 6$ near
$\supp 1-\chi_1$.
 Choose $\chi\in C_c^\infty(X;[0,1])$ so that $\chi \chi_1=\chi_1$.
Then for $z,z_0\in \zhath$, 
\[\begin{split}
\chi R(z)\chi &- \chi R(z_0)\chi= (\pr(z)-\pr(z_0))\chi R(z)\chi \chi_1(2-\chi_1)\chi R(z_0)\chi \\
&+ (1-\chi_1-\chi R(z)\chi[h^2\Delta,\chi_1]) \left( \chi R_0(z)\chi-\chi R_0(z_0)\chi \right) (1-\chi_1+[h^2\Delta, \chi_1]\chi R(z_0)\chi).
\end{split}\]
\end{lem}
It is important to note in the identity above that
 $\chi R_0 \chi$ only appears where it is multiplied both on the left and 
right by an operator (either $1-\chi_1$ or 
$[h^2\Delta, \chi_1]$) supported in the set where $r \ge 6$.  If we think
of this set as a subset of $X_0 = [0,\infty)\times Y$, then the appearance of
 $\chi R_0\chi$ makes sense.

We omit the proof of Lemma \ref{l:vodev} because it is essentially the same as that of \cite[(5.4)]{v} (see also \cite[Lemma 6.26]{dz} and, for another version in the setting of cylindrical ends, \cite[Lemma~2.1]{cd}).

The proof we give of the following theorem follows the proof of 
\cite[Theorem 1.5]{v}, but we write it out in detail because it is short and to highlight the role of the estimates we proved in \S\ref{s:modelcyl}. 
\begin{thm}\label{thm:resonancefree}
With $\chi$ as in Lemma \ref{l:vodev}, using \eqref{e:2n-1} take constants $C$ and $\mu(h)$ such that 
$$\| \chi R(E\pm i0) \chi \|_{L^2(X)\rightarrow L^2(X)} \leq \frac{C}{\mu(h)},$$
where $E = E_0$ and $0<\mu(h)\leq h^{2}$.  Then there are constants 
$C'$, $\tilde{C}$  so that  for $h>0$ sufficiently small, $\chi R(z)\chi$ is analytic in $\{ z\in \zhath: d_h(z,E\pm i0)< C' \mu(h)\}$.  Moreover, in
this region the cutoff resolvent satisfies the estimate
$$ \| \chi R(z)\chi \|_{L^2(X)\to L^2(X)} \leq \frac{\tilde{C}}{\mu(h)},$$
with $\tilde{C}$ depending on $\chi$.
\end{thm}
\begin{proof}
We use the identity from Lemma \ref{l:vodev}, with $z_0=E\pm i0$.  Rearranging, we find (all norms here are $L^2(X) \rightarrow L^2(X)$)
\begin{align*}
\| \chi R(z)\chi\| & \leq \|  \chi R(E\pm i0)\chi\| + 2|\pr(z)-E| \| \chi R(z)\chi\|\| \chi R(E\pm i0)\chi \| \\ & 
+ \| (1-\chi_1)(\chi R_0(z)\chi-\chi R_0(E\pm i0)\chi)(1-\chi_1)\|   \\ & + \| \chi R(z)\chi\| \| [h^2\Delta,\chi_1]) (\chi R_0(z)\chi-\chi R_0(E\pm i0)\chi )(1-\chi_1)\| \\& + \| (1-\chi_1) \left(\chi(R_0(z)\chi-\chi R_0(E\pm i0)\chi\right)
[h^2\Delta, \chi_1\| \| \chi R(E\pm i0)\chi \|  \\ & + \| \chi R(z)\chi\| \| \chi R(E\pm i0)\chi\| \| [h^2\Delta,\chi_1] \left( \chi R_0(z)\chi-\chi R_0(E\pm i0)\chi \right) [h^2\Delta, \chi_1]\| .
\end{align*}
By writing this bound in this detailed fashion we hope to indicate the importance of the improved estimate (\ref{eq:resest2}) as compared to (\ref{eq:resest1}), 
so that, for example, \begin{multline}
\| [h^2\Delta,\chi_1] (\chi R_0(z)\chi-\chi R_0(E\pm i0)\chi )(1-\chi_1)\| 
\\ = \| [h^2\Delta,\chi_1] (R_0(z)\chi- R_0(E\pm i0)\chi )(1-\chi_1)\| \leq Cd_h(z,E\pm i0)/h .
\end{multline}
Using the bound on $\| \chi R(E\pm i0)\chi\| $ from the assumptions along with bounds of Proposition \ref{p:resest}, we find
\begin{align*} 
\| \chi R(z)\chi\| & \leq \frac{C}{\mu(h)} +\frac{C d_h(z,E\pm i0)} { \mu(h) }  \| \chi R(z)\chi\| +  \frac{Cd_h(z,E\pm i0)}{h \mu(h) } \\ & 
+ C d_h(z,E\pm i0) \left( \frac{1}{h}+\frac{1}{\mu(h)}\right)\| \chi R(z)\chi \|.
\end{align*}
Here we have also bounded $|\pr(z)-E |\leq d_h(z,E\pm i0)$, which is weaker than  the estimate from Lemma \ref{l:projdiff} since we will have
$d_h(z,E\pm i0)=O(\mu (h))$.
If we choose $C' $ sufficiently small, the coefficients of $\| \chi R(z)\chi\|$ on the right hand side above will be small enough that the terms with 
$\|\chi R(z)\chi\|$ can be absorbed in the left hand side, proving the result.
\end{proof}

\noindent\textbf{Acknowledgements.} The authors are grateful to Semyon Dyatlov, Plamen Stefanov, and Maciej Zworski for helpful discussions, and to the anonymous referees for many helpful comments, corrections, and suggestions. Thanks also to the Simons Foundation for support through the Collaboration Grants for Mathematicians program. KD was partially supported by the National Science Foundation grant DMS-1708511.

\end{document}